\documentclass[a4paper,oneside]{amsart}
\usepackage{comment}
\usepackage{geometry}
\usepackage{amsmath}
\usepackage{flexisym}
\usepackage{breqn}
\usepackage{amsfonts}
\usepackage{amssymb}
\usepackage{amsmath}
\usepackage{graphicx} 
\usepackage{hyperref}
\usepackage{todonotes}
\hypersetup{
	colorlinks,
	linkcolor={blue!50!black},
	citecolor={blue!50!black},
	urlcolor={blue!50!black}
}
\usepackage{cleveref}
\let\oldlb\(
\let\oldrb\)
\renewcommand{\(}{\begin{dmath*}}
\renewcommand{\)}{\end{dmath*}}
\newtheorem{theorem}{Theorem}[section]

\newtheorem{proposition}[theorem]{Proposition}
\newtheorem{corollary}[theorem]{Corollary}
\newtheorem{definition}[theorem]{Definition}

\newtheorem{example}[theorem]{Example}
\newtheorem{question}[theorem]{Question}

\theoremstyle{definition}
\newtheorem{remark}{Remark}
\newcommand{\N}{\mathbb{N}}
\newcommand{\R}{\mathbb{R}}
\newcommand{\Z}{\mathbb{Z}}

\newcommand{\SE}{\mathbb{S}}
\newcommand{\A}{\mathcal{A}}

\newcommand{\U}{\mathcal{U}}
\newcommand{\upperent}{\overline{\operatorname{ent}}}
\newcommand{\lowerent}{\underline{\operatorname{ent}}}
\newcommand{\sep}{\operatorname{sep}}
\newcommand{\spa}{\operatorname{spa}}
\title{Topological slow entropy, sequence entropy, and generalized $[T,T^{-1}]$ systems}
\author{Nicanor Carrasco-Vargas}
\thanks{\textit{Address.} Faculty of Mathematics and Computer Science, Jagiellonian University, Krak\'ow, Poland.}
\thanks{\textit{Email.} \texttt{nicanor.vargas@uj.edu.pl}}
\begin{document}
\maketitle
\begin{abstract}
We consider topological dynamical systems given by skew products $S\rtimes_{\tau} T$, where $S\colon Y\to Y$ is a subshift, $\tau\colon Y\to\Z$ is a continuous cocycle, and $T$ is an arbitrary invertible topological system. For fixed $(Y,S,\tau)$ it may happen that all systems of the form $S\rtimes_\tau T$ have the same topological entropy, and thus it arises the problem of distinguishing two such systems.

We show that if $T_1$ and $T_2$ are invertible topological dynamical systems with different topological entropy then $S\rtimes_{\tau} T_1$ and $S\rtimes_{\tau} T_1$ can be distinguished using slow entropy as introduced by Katok and Thouvenot. We prove a similar result under the assumption that the fiber systems have different slow entropy at some scale (this can be applied if $T_1$ and $T_2$ have both zero entropy, or have the same entropy). These results 
require rather mild assumptions on $(Y,S,\tau)$, and can be applied to some entropy-zero systems in the base. 

We generalize classical results in the theory of sequence entropy, which were proved by Goodman with  the additional assumption of finite topological dimension. We show that the measure-theoretic sequence entropy of any system is bounded by its topological sequence entropy. Under an extra assumption on the sequence we establish a variational principle, and prove that the topological sequence entropy of a system equals its topological entropy multiplied by a positive constant that depends only on the sequence. 
\end{abstract}
\section{Introduction}
Let $S$ and $T$ be measure preserving transformations of the probability spaces  $(Y,\mathcal C,\nu)$ and $(X,\mathcal B,\mu)$ respectively. We require $T$ to be invertible. Let  $\tau\colon Y\to\Z$ be a measurable function, and define a new  measure preserving transformation $S\rtimes_\tau T$ acting on  $(Y\times X,\mathcal C\times \mathcal B,\nu\times\mu)$  by 
\begin{equation}\label{TTmenos1}(y,x)\mapsto (S(y),T^{\tau(y)}(x)).\end{equation}
The transformation  $S\rtimes_\tau T$ is usually called generalized $[T,T^{-1}]$ system with base $S$ and fiber $T$.  The classical $[T,T^{-1}]$ system is obtained by taking $T$ and $S$ as the Bernoulli shift over $\{-1,1\}^\Z$, with the uniform measure, and $\tau$  given by $\tau(y)=y(0)$ for all $y\in Y$. The space consists of pairs $(y,x)$, $x,y\in \{-1,1\}^\Z$. The transformation shifts  $y$, and applies either the shift or the inverse of the shift to $x$, depending on whether $y(0)$ equals $1$ or $-1$. In the famous work \cite{kalikow_t_1982} Kalikow proved that this system is not loosely Bernoulli despite being $K$, and it is generally regarded as the first natural system with this property.  Generalized $[T,T^{-1}]$ systems have been subsequently studied by numerous authors \cite{katok_smooth_nodate,rudolph_asymptotically_1988,den_hollander_mixing_1997,hollander_random_2006,dolgopyat_mixing_2022, ball_entropy_2003, heicklen_entropy_2000, dolgopyat_flexibility_2022, austin_scenery_2014}. 

If we fix the tuple $(Y,\mathcal C,\nu,S,\tau)$ then we have a large family of systems of the form $S\rtimes_\tau T$, and we encounter the problem of distinguishing them. For instance if we take $S$ as the Bernoulli shift over $\{-1,1\}^\Z$ or $\{-1,1\}^\N$ and $\tau(y)=y(0)$, then Abramov and Rokhlin's formula \cite{abramov_entropy_1965} shows that $S\rtimes_\tau T$ has entropy $\log(2)$ for every choice of $(X,\mathcal B,\mu, T)$. Since entropy provides no information for this family of systems, the following question arises naturally:
\begin{question}\label{question}
Does an isomorphism between $S\rtimes_\tau T_1$ and $S\rtimes_\tau  T_2$  imply that $T_1$ and $T_2$ have the same measure theoretic entropy?
\end{question} 
In other words, this question asks whether the entropy of the fiber is an  invariant for isomorphism within the family of systems of the form $S\rtimes_\tau T$, for fixed $S$. This question was answered positively in \cite{heicklen_entropy_2000} in the case where $S$ is  the shift map over $\{-1,1\}^\N$ with the uniform measure and $\tau$ as the zero coordinate. This was later extended to the invertible case in \cite{austin_scenery_2014}. The results in \cite{austin_scenery_2014} are more general as they allow to consider a subshift of finite type in the base  under certain technical assumptions. Related works include \cite{aaronson_relative_2012,ball_entropy_2003}. Not much is known about \Cref{question} when the system in the base has low complexity or intermediate complexity. 

In this work we study the topological counterpart of $[T,T^{-1}]$ systems and the topological counterpart of \Cref{question}. Our setting is rather general in the sense that the system in the base may have zero or positive entropy and both situations are treated in a unified manner. 
\subsection{Results}
Consider a tuple $(Y,S,\tau)$, where $Y\subset \A^\Z$ is a subshift with finite alphabet $\A$, $S\colon Y\to Y$ is the shift transformation over $Y$, and $\tau\colon Y\to \Z$ is a continuous function (i.e. a continuous cocycle). For every invertible topological dynamical system $(X,T)$ we consider the continuous transformation $S\rtimes_\tau T$ over $Y\times X$ defined by \Cref{TTmenos1}. The analog of \Cref{question} asks the following.
\begin{question}\label{question-topological}
    Does a topological conjugacy between $S\rtimes_\tau T_1$ and $S\rtimes_\tau T_2$ imply that $T_1$ and $T_2$ have the same topological entropy?
\end{question}
In other words, this question asks whether the topological entropy of the fiber is a conjugacy invariant within the family of systems of the form $S\rtimes_\tau T$, for fixed $S$.
 
Our main result is a positive answer to \Cref{question-topological} 
under rather mild assumptions on $(Y,S,\tau)$. We show that in fact the topological entropy of $T$ equals the topological \textit{slow entropy} of $S\rtimes_\tau T$ at a suitable scale which only depends on $(Y,S,\tau)$. Let $\textbf{a}=\{a_n(t)\}_{n\in\N,t>0}$ be an arbitrary family of functions increasing to infinity and monotone in $t$,  we call it a scale. The upper topological slow entropy $\upperent_{\textbf{a}}(T)$ of a topological dynamical system $(X,T)$ with scale $\textbf{a}$ is a real-valued invariant of conjugacy. It is obtained by comparing $a_n(t)$ with the number of $n$-Bowen balls of certain radius that are needed to cover $X$. This invariant was first introduced in the measure theoretic setting by Katok and Thouvenot  \cite{AIHPB_1997__33_3_323_0}, see also the survey \cite{gogolev_survey_2024}. One similarly defines the lower slow entropy $\lowerent_{\textbf{a}}(T)$, which is not always equal to the upper slow entropy.

In the next result we assume that $\tau$ has a property which we call being \textit{$\lambda$-unbounded} for some $\lambda>0$ (\Cref{def:good}). This condition informally says that for all $N$ the function $n\mapsto\sum_{i=0}^{n-1}(S^{i}y)$ has at least $N$ different values in its range, for a positive ``proportion'' $\lambda$ of elements $y\in Y$. This is stated as a condition on the language of the subshift and makes no reference to a measure. 

\begin{theorem}\label{thm:entropy}
Let $(Y,S)$ be a subshift, and let $\tau\colon Y\to \Z$ be a continuous map which is $\lambda$-unbounded for some $\lambda>0$. Then there exists a scale $\textbf{a}=\{a_n(t)\}_{n\in\N,t>0}$ depending only on $(Y,S,\tau)$, and such that
\[\upperent_{\textbf{a}}(S\rtimes_\tau T)=\lowerent_{\textbf{a}}(S\rtimes_\tau T)=h_{top}(T)\]  for every invertible topological dynamical system $(X,T)$.
\end{theorem}

  Let us remark that \Cref{question-topological} is most interesting for a tuple $(Y,S,\tau)$ with the property that all the skew products $S\rtimes_\tau T$ have the same topological entropy as $S$. Informally speaking, this means that the number of $\epsilon,\{0,\dots,n-1\}$-Bowen balls associated to $S\rtimes_\tau T$ needed to cover $Y\times X$ can be estimated as
\[\spa(S\rtimes T,\{0,\dots,n-1\},\epsilon)\approx e^{nh_{top}(S)+f(n)h_{top}(T)}.\]
for a sublinear function $f(n)$. One intuitive interpretation of \Cref{thm:entropy} is that the scale $a_n(t)$ isolates the contribution of $T$ (or $h_{top}(T)$).

With the extra assumption that $\tau$ only attains values in $\{-1,0,1\}$ then we can prove a more general result, where the topological entropy of the fiber is replaced by slow entropy at an arbitrary scale. The main reason behind this assumption is that it guarantees that every set $\{\sum_{i=1}^k \tau(S^i y ) : k=1,\dots,n-1\}$, $y\in Y$, $n\in\N$ is in fact an interval, and this simplifies matters significantly (see \Cref{sec:estimates}).
\begin{theorem}\label{thm:slow-entropy}
Let $(Y,S)$ be a subshift, and let $\tau\colon Y\to \Z$ be a continuous map which is $\lambda$-unbounded for some $\lambda>0$. Suppose further that $\tau(y)\in\{-1,0,1\}$ for every $y\in Y$. Let $\textbf{b}=\{b_n(t)\}_{n\in\N,t>0}$ be an arbitrary scale. Then there exists a scale $\textbf{c}=\{c_n(t)\}_{n\in\N,t>0}$ such that
\[\upperent_{\textbf{c}}(S\rtimes_\tau T)\leq\upperent_{\textbf{b}}(T)\]  \[\lowerent_{\textbf{c}}(S\rtimes_\tau T)\geq \lowerent_{\textbf{b}}(T)\] 
for every invertible topological system $(X,T)$.
\end{theorem}
The conclusion of this result is most meaningful when the upper and lower topological slow entropy of $T$ with scale $b_n(t)$ are equal. In this case the upper and lower topological slow entropy of $S\rtimes_\tau T$ at scale $c_n(t)$ are both equal to the slow entropy of $T$ with scale $b_n(t)$.

The next Corollary summarizes the relevant consequence of these two results in relation to \Cref{question-topological}.
\begin{corollary}\label{corollary}
    Let $(Y,S)$ be a subshift and let $\tau\colon Y\to \Z$ be a continuous map which is $\lambda$-unbounded for some $\lambda>0$. Let $T_1$ and $T_2$ be invertible topological dynamical systems and assume that there is a topological factor map from $S\rtimes_\tau T_1$ to $S\rtimes_\tau T_2$. Then $h_{top}(T_1)\geq h_{top}(T_2)$. 

    With the extra assumption that $\tau(y)\in \{-1,0,1\}$ for every $y\in Y$, we have $\upperent_{\textbf{b}}(T_1)\geq \lowerent_{\textbf{b}}(T_2)$ for every scale $\textbf{b}$.  
\end{corollary}
This result is a topological generalization of the results in \cite{
heicklen_entropy_2000,ball_entropy_2003,austin_scenery_2014}. As we mentioned before, known results for measure-preserving transformations require the system in the base to have positive entropy ($\{-1,1\}^\N$ with the uniform measure in \cite{
heicklen_entropy_2000,ball_entropy_2003}, and a mixing $\Z$-SFT in \cite{austin_scenery_2014}). In contrast, the result above can be applied to zero-entropy systems in the base. For instance, this result applies to any minimal subshift $(Y,S)$ with a cocycle $\tau$ which is not a coboundary.  We shall review this and other examples in \Cref{sec:examples}.

It is natural to ask  if the scales that we construct can be used for measure-theoretic $T,T^{-1}$ systems. This is not the case in general, in the sense that the scale may not capture the entropy of the fiber. For instance, consider a skew product $S\rtimes_\tau T$, where $S$ is the shift over $Y=\{-1,1\}^\Z$, $\tau$ is given by $\tau(y)=y(0)$ for all $y$, and $(X,T)$ is an arbitrary invertible system. Endow $Y\times X$ with the product $\nu\times \mu$ of the uniform measure $\nu$ on $\{-1,1\}^\Z$, and a $T$ invariant Borel probability measure $\mu$ on $X$ with finite entropy. It is not difficult to verify that the measure-theoretic slow entropy of $S\rtimes_\tau T$ with the measure $\nu\times\mu$ and the scale $a_n(t)$ from \Cref{thm:entropy} is zero. This means that the correct estimate for Hamming (pseudo)-balls associated to $S\rtimes_\tau T$ grows much slower than the one for Bowen balls. This is a different situation from \cite{kanigowski_slow_2019,kanigowski_slow_2018,kanigowski_slow_2022}, where the two estimates coincide and topological slow entropy can be used to compute measure-theoretic slow entropy. The possibility of results analogous to \Cref{thm:entropy} and \Cref{thm:slow-entropy} for measure-preserving systems will be studied in future work. 

Another invariant considered in this work is sequence entropy. This invariant is obtained by replacing in the definition of entropy the iterates $T^1,T^2,T^3,\dots$ by $T^{t_1},T^{t_2},T^{t_2},\dots$, for an arbitrary sequence of natural numbers $A=(t_i)_{i\geq 1}$. The invariant obtained is denoted $h^A_{top}(T)$ for a topological system $(X,T)$, and $h^A_{\mu}(T)$ for a measure preserving system $(X,\mathcal B,\mu, T)$. This invariant was first introduced for measure preserving systems in \cite{kushnirenko_metric_1967}, and was later extended to topological systems in \cite{goodman_topological_1974}. Subsequent works in this subject include \cite{gao_sequence_2020, canovas_interval_2000,hric_topological_nodate}. 

In this work we strengthen a classical result in the theory of sequence entropy. Namely, the next result was proved in the seminal work of Goodman  \cite[Theorems 4.1 and 4.4]{goodman_topological_1974} under the extra assumption of finite topological dimension. 
\begin{theorem}\label{Topological-kruw-newton}\label{goodman-v2}
Let $X$ be a compact metric space and let $T\colon X\to X$  be continuous  transformation. Let $A=\{t_1,\dots\}$ be a sequence of natural numbers, and define
\[S_A(n,m)=\{t_i+j : i=1,\dots,n,\ j=0,\dots,m-1\}, \ \ n,m\in\N\]
\begin{equation}\label{def:K}
K(A)=\lim_{m\to\infty} \limsup_{n\to\infty} \frac{1}{n}|S_A(n,m)|.
\end{equation}
Then:
\begin{enumerate}
    \item Goodwyn's theorem holds for sequence entropy:
    \[\sup_{\mu\in M_T(X)} h^{A}_{\mu}(T)\leq h^A_{top}(T).\]
    \item Unless ($K(A)=\infty$ and $h_{top}(T)=0$), the variational principle holds for sequence entropy:
    \[\sup_{\mu\in M_T(X)} h^{A}_{\mu}(T) = h^A_{top}(T).\]
    \item  Unless ($K(A)=\infty$ and $h_{top}(T)=0$), 
    \[h^A_{top}(T)= \begin{cases}
        K(A)h_{top}(T) &K(A)<\infty, \  h_{top}(T)<\infty\\
        \infty & K(A)=\infty, \ h_{top}(T)>0\\
        \infty & K(A)>0, \ h_{top}(T)=\infty\\
        0&  K(A)=0,  \ h_{top}(T)=\infty.
    \end{cases}
    \]
\end{enumerate}
Here $M_T(X)$ stands for the set of $T$-invariant Borel probability measures on $X$.
\end{theorem}
\begin{remark} 
    Part (1) in \Cref{goodman-v2} was proved by Goodman \cite[Theorem 3.1]{goodman_topological_1974}, by Eberlein \cite[Theorem 4.3]{eberlein_topological_nodate}, and more recently by Huang and Ye \cite[Theorem 2.6]{huang_combinatorial_2009}. These proofs rely on the relation $h^{A}_{\mu}(T\times T)=2 h^{A}_{\mu}(T)$, stated in Kushnirenko's work \cite{kushnirenko_metric_1967}. This relation was  disproved in \cite{lemanczyk_sequence_1985} (see also \cite[Section 4]{gogolev_survey_2024}). To our knowledge the proof presented here is the first to avoid this issue. See also \Cref{remark:explanation-goodman-proof}.
\end{remark}

There are families of systems for which both sequence entropy and slow entropy have shown to be useful invariants, see for instance \cite{kanigowski_slow_2019}. This may give the impression that both invariants have similar distinguishing power. In \Cref{example:slow-entropy-and-sequence-entropy} we find a family of systems where sequence entropy provides no information, but slow entropy does.

\subsection*{Acknowledgements}
The author thanks A. Kanigowski for his guidance, helpful questions, and for kindly reading preliminary versions of this text. This work was supported by a grant from the Priority Research Area SciMat under the Strategic Programme Excellence Initiative at Jagiellonian University.

\section{Preliminaries}\label{sec:preliminaries}
In this section we review definitions and basic facts about sequence entropy and topological slow entropy. The reader is referred to the survey \cite{gogolev_survey_2024} for more details.

Let $(X,T)$ be a topological dynamical system. By this we mean that $T\colon X\to X$ is a continuous transformation of the compact metric space $(X,d)$. We do not require $T$ to be invertible.  We let $\SE=\Z$ if $T$ is invertible and $\SE=\N$ otherwise. We write $F\Subset \SE$ when $F$ is a finite subset of $\SE$. The notation $F+F'$ stands for $\{f+f' : f\in F, f'\in F'\}$, $F+n$ denotes  $\{f+n : f\in F\}$, $F,F'\Subset \SE$, $n\in\SE$, and $|F|$ denotes the cardinality of $F$. We call $F$ an interval if it is equal to $\{n,\dots,m\}$ for some $n,m\in\SE$.
\subsection{Bowen metrics and open covers}
For each $F\Subset \SE$ we define the Bowen metric $d^B_F$ on $X$ by
\[d^B_{F}(x,x')=\max\{d(T^{i}(x),T^{i}(x'): i\in F\}.\]
The Bowen $F,\epsilon$-ball is  $B^B_F(x,\epsilon)=\{y\in X: d_F^B(x,y)<\epsilon\}$. If $d^B_{F}(x,x')\leq \epsilon$, we say that $x$ and $x'$ are $\epsilon,F$-close. A set $E\subset X$ is called $\epsilon,F$-separated if $d^B_F(x,y)>\epsilon$ for every $x,y \in E$, and $\epsilon,F$-spanning if for all $x\in X$ there is $y\in E$ with $d^B_F(x,y)\leq\epsilon$. $\sep(T,F,\epsilon)$ denotes the maximal cardinality of an $\epsilon,F$-separated set, and  $\spa(T,F,\epsilon)$ denotes the minimal cardinality of a $\epsilon,F$-spanning set. One easily verifies the inequalities
\begin{equation}\label{sep-and-spa}\spa(T,F,\epsilon)\leq \sep(T,F,\epsilon)\leq \spa(T,F,\epsilon/2), \ \ \epsilon>0, \ F\Subset\SE.
\end{equation}  
We now review some basic facts about finite open covers. A cover $\U$ of $X$ is a finite collection of open subsets of $X$ whose union equals $X$. Given two covers $\U$ and $\mathcal V$ we write $\U\vee \mathcal V=\{U\cap V : U\in \U, V\in \mathcal V\}$.  We denote by $N(\U)$ the minimal cardinality of a subset of $\U$ that covers $X$, called a subcover. We write \[\U^F = \bigvee_{i\in F}T^{-i}\U, \ \ F\Subset \SE.\] 
\begin{proposition}\label{wellknown}
Fix a finite open cover $\U$ of $X$ and $\epsilon>0$. If every element in $\U$ has diameter smaller than $\epsilon$ then
    \[\spa(T,F,\epsilon)\leq N(\U^F) \ \text{for all }F\Subset \SE.\]
If $\epsilon$ is smaller than the  Lebesgue number for $\U$ then
    \[N(\U^F)\leq\sep(T,F,\epsilon) \ \text{for all }F\Subset \SE.
    \]
\end{proposition}See \cite[Lemma 9.37]{kerr_ergodic_2016} for a proof. Alternatively, the proof in the case $F=\{0,\dots,n-1\}$ is well known and can be found in \cite[Proposition 3.8, Chapter 6]{petersen_ergodic_1983}. The same argument can be applied to  $F\Subset\SE$. 

\begin{proposition}\label{monotony-properties-of-spa}Let $F,F'\Subset\SE$ and $n\in\SE$. 
    \begin{enumerate}
        \item If $F\subset F'$ then $\spa(T,F,\epsilon)\leq\spa(T,F',\epsilon)$. 
        \item We have $\spa(T,F,\epsilon)=\spa(T,F+n,\epsilon)$. 
        \item If $F$ and $F'$ are intervals and $|F|\leq |F'|$ then $\spa(T,F,\epsilon)\leq \spa(T,F',\epsilon)$
    \end{enumerate}
    \end{proposition}
\begin{proof}
    The first two claims follow directly from the definitions. The third follows from (1) and (2), since $|F|\leq |F'|$ implies that for some $n$ we have $F+n\subset F'$.
\end{proof}
\subsection{Topological entropy and sequence entropy} Let $A=(t_i)_{i\geq 1}$ be a sequence of natural numbers. The topological sequence entropy $h^A_{top}(T)$ is defined by 
\[ h^{A}_{top}(T,\U)=\limsup_n \frac{1}{n}\log(N( \U^{\{t_1,\dots,t_n\}})) \]
\[
h^A_{top}(T)=\sup_{\U}h^A_{top}(T,\U).\]
Here $\U$ ranges over finite open covers of $X$. Topological sequence entropy can also be computed in terms of $\spa(T,F,\epsilon)$. For this we define 
\begin{equation}\label{definition-h-T-epsilon}h^A_{top}(T,\epsilon)=\limsup_{n} \frac{1}{n}\log \spa(T,\{t_1,\dots,t_n\},\epsilon), \ \ \ \epsilon>0.\end{equation}
\begin{proposition}\label{sequence-entropy-spanning-sets}
    $h^A_{top}(T)=\sup_{\U}h^A_{top}(T,\U)=\lim_{\epsilon\to 0^+} h^A_{top}(T,\epsilon)$
\end{proposition}
\begin{proof}
\Cref{wellknown} shows that for every $\epsilon$ there is a finite open cover $\U$ such that $h^A_{top}(T,\epsilon)\leq h^A_{top}(T,\U)$, and that for every $\U$ there is $\epsilon$ small enough such that $h^A_{top}(T,\U)\leq h^A_{top}(T,\epsilon)$.
\end{proof}
In the case $A=\{0,1,2,\dots\}$ we omit $A$ from the notations and $h^A_{top}(T)=h_{top}(T)$ is called the topological entropy of $T$. 
\subsection{Measure theoretic entropy and sequence entropy}
Let $\mu$ be a Borel $T$-invariant probability measure on $X$. If $\xi$ is a finite Borel measurable partition of $X$ then we write $H(\xi)=\sum_{P\in \xi}-\mu(P)\log \mu(P)$. If $\eta$ is another partition then $\xi\vee \eta $ denotes $\{P\cap Q : P\in \xi, Q\in \eta\}$. We write
\[\xi^F = \bigvee_{i\in F}T^{-i}\xi, \ \ F\Subset \SE.\] 

The measure-theoretic sequence entropy $h^{A}_{\mu}(T)$ is defined by
\[h^{A}_{\mu}(T)=\limsup_n \frac{1}{n}H(\xi^{\{t_1,\dots,t_n\}}),\]
\[h^{A}_{\mu}(T)=\sup_{\xi} h^{A}_{\mu}(T,\xi)\]
Here $\xi$ ranges over all finite Borel-measurable partitions of $X$. In the case where the sequence $A$ is given by $t_n=n-1$ we omit $A$ from the notation and $h^{A}_{\mu}(T)=h_{\mu}(T)$ is called the measure theoretic entropy of $T$ associated to $\mu$. 
\subsection{Topological slow entropy}
Let $\textbf{a}=\{a_n(t)\}_{n\in\N,t>0}$ be a family of positive sequences increasing to infinity and monotone in $t$. We say that $\{a_n(t)\}_{n\in\N,t>0}$ is a scale. The topological slow entropy $\upperent_{\textbf{a}}(T)$ of $T$ with respect to $\{a_n(t)\}_{n\in\N,t>0}$ is defined by  
\[\upperent_{\textbf{a}}(T,\epsilon)=\sup(\{0\}\cup\{t>0 : \limsup_n\frac{\spa(T,\{0,\dots,n-1\},\epsilon)}{a_n(t)}>0\}), \  \ \epsilon>0\]
\[\upperent_{\textbf{a}}(T)=\lim_{\epsilon\to 0} \upperent_{\textbf{a}}(T,\epsilon)\]
The lower topological slow entropy $\lowerent_{\textbf{a}}(T)$ is defined similarly, replacing $\limsup$ by $\liminf$.

One easily verifies that $\upperent_{\mathbf{a}}$ and $\lowerent_{\mathbf{a}}$ are conjugacy invariants, and non-increasing by factor maps.
\begin{proposition}\label{prop:slow-entropy-with-exponential-scale-is-entropy}
The topological entropy $h_{top}(T)$ is equal to the upper and lower topological slow entropy of $T$ with respect to   the scale $\textbf{e}=\{e^{nt}\}_{n\in\N,t>0}$. That is, 
\[\upperent_{\mathbf{e}}(T)=\lowerent_{\mathbf{e}}(T)=h_{top}(T).\]
\end{proposition}
\begin{proof}
We start by proving that $\upperent_{\mathbf{e}}(T)\leq h_{top}(T)$. This is clearly true if $h_{top}(T)=\infty$. Otherwise, let $t$ be an arbitrary real number with $h_{top}(T)<t$. Choose $\epsilon$ small enough so that $h_{top}(T,\epsilon)<t$, and let $t'$ satisfy $h_{top}(T,\epsilon)<t'<t$. Since $h_{top}(T,\epsilon)$ is defined as the limsup of $\frac{1}{n}\log(\spa(T,\{0,\dots,n\},\epsilon)$ and this limsup is strictly smaller than $t'$, it follows that for all $n$ large enough we have $\spa(T,\{0,\dots,n-1\},\epsilon)\leq e^{nt'}$. Therefore
\[\limsup_{n\to\infty} \frac{\spa(T,\{0,\dots,n-1\},\epsilon)}{e^{nt}}\leq \limsup_{n\to\infty} \frac{e^{nt'}}{e^{nt}}=\lim_{n\to\infty}e^{n(t'-t)}=0.\]
According to the definition of upper slow entropy, this means that  $\upperent_{\mathbf{e}}(T,\epsilon)\leq t$. Since $t$ was an arbitrary value with $h_{top}(T)<t$, our claim that $\upperent_{\mathbf{e}}(T)\leq h_{top}(T)$ follows.  For the remaining inequality $h_{top}(T)\leq \lowerent_{\mathbf{e}}(T)$ one follows a similar argument, but instead one uses the fact that $h_{top}(T)$ equals the limit as $\epsilon\to^+0$ of $\liminf_{n\to\infty}\frac{1}{n}\log(\spa(T,\{0,\dots,n\},\epsilon)$.\end{proof} 
\subsection{Subshifts}\label{subsec:subshifts} A subshift on $\SE$ is a subset  $Y\subset \A^{\SE}$ which is topologically closed in the prodiscrete topology, and invariant for the shift map $\sigma\colon \A^{\SE}\to \A^{\SE}$,
\[x\mapsto \sigma(x), \ \sigma(x)(n)=x(n+1), \ \ \ n\in\SE.\]
A pattern with support $F\Subset\SE$ is a map $w\colon F\to \A$. We write $L_F(Y)=\{y|_F : y\in Y\}$. We also write $w\sqsubset y$ when the restriction of $y$ to $F$ equals $w$. The pattern $w$ defines the cylinder $[w]=\{y\in Y : w\sqsubset y\}$. A metric for the prodiscrete topology on $Y$ is given by
\[d(x,y)=\inf(\{2^{-n} : x(i)=y(i) \ \forall i\in \SE, \ |i|\leq n\} \cup \{2\})\]
We call this the standard metric for subshifts. 

Given $n\geq 1$ we write $L_n(Y)=L_{\{0,\dots,n-1\}}(Y)$. Given $s\geq 0$   the notation $L_{n,s}(Y)$ stands for $L_{\{-s,\dots,s+n-1\}}(Y)$ when $\SE=\Z$, and $L_{\{0,\dots,s+n-1\}}(Y)$ when $\SE=\N$.

\section{Sequence entropy and proof of \Cref{goodman-v2}}
In this section we prove \Cref{goodman-v2}. It will be convenient to fix for the rest of the section the objects with which we will work. Fix a compact metric space $(X,d)$, a (possibly non-invertible) continuous transformation $T\colon X\to X$,  a Borel $T$-invariant probability measure $\mu$ on $X$, and a sequence $A=(t_i)_{i\in\N}$ of natural numbers.

\begin{definition}
    For $F\Subset\N$, $\epsilon>0$, and $\delta\in (0,1)$, let $\spa(T,F,\epsilon,\delta)$ be the minimal number of Bowen $\epsilon,F$-balls needed to cover a subset of $X$ with measure $\mu$ at least $1-\delta$.
\end{definition}
A. Katok proved \cite[Theorem 1.1]{katok_lyapunov_1980} that if $\mu$ is ergodic and $T$ is invertible, then the measure-theoretic entropy $h_{\mu}(T)$ can be computed as 
\[h_\mu(T)=\lim_{\epsilon\to 0^+} \left(\limsup_n \frac{1}{n}\log \spa(T,\{0,\dots,n-1\},\epsilon,\delta)\right)\]
for any fixed  $\delta\in(0,1)$. This formula fails without assuming ergodicity. For instance, if $h_{\mu}(T)>0$  and $X$ has a subset with measure $\delta>0$ where $T$ acts as the identity function, then the formula at the right side will be $0$ and in particular different from $h_\mu(T)$.

In the next result we show that without the ergodicity assumption, if we take the limit $\delta\to 0$ in the formula above, then the inequality $\leq$ is still true. This holds in the more general context of sequence entropy. 
\begin{proposition}\label{katok-v2}
Let $\mu$ be a $T$-invariant Borel probability measure on $X$.  
Then the measure-theoretic sequence entropy $h^A_{\mu}(T)$ satisfies
    \[h^{A}_{\mu}(T)\leq \lim_{\delta\to^+0}\lim_{\epsilon\to 0^+}\left(\limsup_n \frac{1}{n}\log \spa(T,\{t_1,\dots,t_n\},\epsilon,\delta)\right).\]
\end{proposition}

In the proof of this result we will repeat some steps from A. Katok's proof in \cite[Theorem 1.1]{katok_lyapunov_1980} with obvious modifications. That is, we replace $\{0,\dots,n-1\}$ by $\{t_1,\dots,t_n\}$ as index set for Bowen and Hamming (pseudo-)metrics. We follow similar notations to simplify the comparison.  However, some important modifications are required to drop the ergodicity assumption. We will need the following elementary result. 
\begin{proposition}\label{partition-with-one-very-big-atom}
    Let $\mu$ be a Borel probability measure on $X$. 
    Let $\delta\in(0,1
    )$. Let $\beta$ be a finite measurable partition of $X$ with $k+1$ elements ($k\geq 1$) and suppose further that there is an atom $P\in\beta$ with $\mu(P)\geq 1-\delta$. Then $H(\beta)\leq 2/e+\delta\log(k)$. 
\end{proposition}
\begin{proof}
    Let $P\in\beta$ be the atom in the statement, and let $\delta_0$ satisfy $1-\delta_0=\mu(P)\geq 1-\delta$.  Let $P_1,\dots,P_k$ be the rest of the atoms so that $\sum_{i=1}^k\mu(P_i)=\delta_0$. The standard argument using Jensen's inequality and the concavity of the map $f(x)=-x\log(x)$ for $x\in [0,1]$ shows that $H(\beta)$ is maximized when all atoms $P_1\dots,P_k$ have equal measure $\delta_0/k$, and therefore \(H(\beta)\leq -\mu(P)\log\mu(P)-\sum_{i=1}^k \delta_0/k\log(\delta_0/k) 
    =-(1-\delta_0)\log(1-\delta_0)+-\delta_0\log(\delta_0)+\delta_0\log(k)\)
    We omit this computation since the argument is well-known. The reader is referred to   \cite[page 231]{petersen_ergodic_1983}. Furthermore, an elementary computation using derivatives shows that the function $f(x)=-x\log(x)$ over $[0,1]$ attains a maximum at $f(1/e)=1/e$ (see \cite[Page 79]{walters_introduction_2000}). Also note that $\delta_0\leq \delta$. Thus the inequality above becomes $H(\beta)\leq 1/e+1/e+\delta\log(k)$. \end{proof}
    We  now review some facts about Hamming (pseudo-) metrics that will be used in the proof of \Cref{katok-v2}. Let $F$ be a finite set and let $n\in\N$. Consider the set  $F^n$ of $n$-tuples of elements in $F$ endowed with the standard Hamming metric \[d^H(u,v)=\frac{1}{n}|\{i\in\{0,\dots,n-1\} : u_i\ne v_i\}|, \ \ \ u=(u_i)_{i=0}^{n-1}, \ v=(v_i)_{i=0}^{n-1}\in F^n.\]
    We simply write $d^H$ without reference to $F^n$ as there is no risk of ambiguity. The corresponding Hamming ball in the metric space $(F^n,d^H)$ is denoted
    \[
    B^H_{F^n}(u,r)=\{v\in F^n : d^H(u,v)<r\}, \ \ \ r>0.
    \]
    Since  the cardinality of this ball is independent of $u$ we may simply write $|B^H_{F^n}(u,r)|=|B^H_{F^n}(r)|$. In \cite[Equation 1.3]{katok_lyapunov_1980} it is proved the following estimate for $0<r<\frac{1}{|F|}(|F|-1)$:  \begin{equation}\label{estimate-for-cardinality-of-hamming-balls}
    \lim_{n\to\infty}\frac{1}{n}\log|B^H_{F^n}(r)|=r\log(|F|-1)-r\log r-(1-r)\log(1-r).
    \end{equation}
Now let $\xi$ be a finite and measurable partition for $X$. We consider the finite set $\xi^n$ with the Hamming metric $d^H$ defined before, and we use this to define a family of pseudometrics on $X$ associated to $T$,  $\xi$,  and the sequence $A=\{t_1,\dots,\}$. 

Given $x\in X$ we denote by  $\xi(x)$  the atom from $\xi$ to which $x$ belongs. We denote by $c^\xi_n(x)$ the ``code'' $(\xi(T^{t_1}x),\dots,\xi(T^{t_n} x))\in\xi^n$. We define the pseudo-metric $d^{H,\xi}_{\{t_1,\dots,t_n\}}$ on $X$  by \[d^{H,\xi}_{\{t_1,\dots,t_n\}}(x,y)=d^H(c^\xi_n(x),c^\xi_n(y)).\] 
    Given $x\in X$ and $r>0$ we define the corresponding Hamming pseudo-ball:
    \[B^{H,\xi}_{\{t_1,\dots,t_n\}}(x,r)=\{y\in X : d^{H,\xi}_{\{t_1,\dots,t_n\}}(x,y)<r\}.\]
    In this manner two elements $x,y$ are close in the pseudo-metric $d^{H,\xi}_{\{t_1,\dots,t_n\}}$ when their iterates $T^{t_1},\dots,T^{t_n}$ are Hamming-close according to $\xi$. 

\begin{proof}[Proof of \Cref{katok-v2}]
    Observe that $h^{A}_{\mu}(T)$ equals the supremum of $h^{A}_{\mu}(T,\xi)$, with $\xi$ ranging over all Borel-measurable finite partitions of $X$ such that $\mu(\partial\xi)=0$, where $\partial \xi$ is the union of the boundaries $\partial P$ of the atoms  $P\in \xi$. Indeed, we can choose $\xi$ such that $\partial \xi=0$ and with diameter as small as we want \cite[Lemma 8.5]{walters_introduction_2000}. This ensures that the supremum is attained by \cite[Lemma 2]{kushnirenko_metric_1967}. We note that \cite[Lemma 2]{kushnirenko_metric_1967} is stated for an invertible transformation, but this assumption is not used in the proof. 
    
    Thus in order to  prove \Cref{katok-v2} it suffices to fix an arbitrary  finite Borel measurable partition $\xi$ for $X$ with $\mu(\partial \xi)=0$ and   prove  \begin{equation}\label{katok-v2-inequality-leq}
        h^{A}_{\mu}(T,\xi)\leq \lim_{\delta\to^+0}\lim_{\epsilon\to 0^+}\left(\limsup_n \frac{1}{n}\log \spa(T,\{t_1,\dots,t_n\},\epsilon,\delta)\right).
\end{equation} 
    Fix $\delta\in(0,1)$, $\epsilon\in (0,\frac{|\xi|-1}{|\xi|})$, and $n\geq 1$. For each $\gamma>0$ we define \[U_\gamma(P)=\{x\in P : \exists y\in X\smallsetminus P, d(x,y)<\gamma)\}, \ P\in \xi\]
    and $U_\gamma(\xi)=\bigcup_{P\in\xi}U_\gamma(P)$.
    Since $\xi$ is chosen with $\mu(\partial\xi)=0$, it follows that 
    \[
    0=\mu(\partial \xi)=\mu(\bigcap _{\gamma>0} U_{\gamma})=\lim_{\gamma\to 0}\mu(U_\gamma)\]
    Therefore we can choose $\gamma=\gamma(\epsilon)$ small enough such that  $\mu(U_\gamma(\xi))<\epsilon^2$ and $\gamma<\epsilon$.       
    
    Let $\chi_\gamma$ be the characteristic function of $U_\gamma$, and let $B_{n,\epsilon}$ be the set of elements in $X$ such that $\sum_{i=1}^n \chi_{\gamma}(T^{t_i}(x))<n\epsilon$. 
    We claim that  $\mu(X\smallsetminus B_{n,\epsilon})<\epsilon$. Indeed, since $T$ preserves $\mu$ and $\sum_{i=1}^n \chi_{\gamma}(T^{t_i}(x))$ is larger than $n\epsilon$ when $x\in X\smallsetminus B_{n,\epsilon}$, we have 
    \(\mu(X\smallsetminus B_{n,\epsilon}) n\epsilon  \leq \int_{X\smallsetminus B_{n,\epsilon}} \sum_{i=1}^n \chi_{\gamma}(T^{t_i}(x)) d\mu(x)
    \leq n \int_{X\smallsetminus B_{n,\epsilon}} \chi_{\gamma} d\mu
    \leq n \mu (U_{\gamma})=n\epsilon^2.\)
    
    Thus $\mu(X\smallsetminus B_{n,\epsilon})<\epsilon$.
    
    Now let  $x,y \in X$ with $d^B_{\{t_1,\dots,t_n\}}(x,y)<\gamma$. The definition of $U_\gamma$ ensures that for every $i=1,\dots,n$ such that $T^{t_i}(x)$ and $T^{t_i}(y)$ belong to a different atom from $\xi$, we must have $T^{t_i}(x),T^{t_i}(y)\in U_\gamma$.  If $x$ is chosen in $B_{n,\epsilon}$ then $T^{t_i}x\in U_\gamma$ is possible for at most $\epsilon n$ values of  $i\in\{1,\dots,n\}$, and for all other values of $i$ we must have $\xi(T^{t_i} x)=\xi(T^{t_i}y)$. Therefore the conditions $d^B_{\{t_1,\dots,t_n\}}(x,y)<\gamma$ and $x\in B_{n,\epsilon}$ imply $d^{H,\xi}_{\{t_1,\dots,t_n\}}(x,y)<\epsilon$.

    The observation in the previous paragraph shows that for every $y\in X$,  the intersection of $B_{n,\epsilon}$ with the Bowen ball 
    $B^B_{\{t_1,\dots,t_n\}}(y,\gamma)$
    is contained in the Hamming pseudo-ball $B^{H,\xi}_{\{t_1,\dots,t_n\}}(y,\epsilon)$. That is,
    \begin{equation}\label{bowen-contained-in-hamming}
    B_{n,\epsilon}\cap B^B_{\{t_1,\dots,t_n\}}(y,\gamma)\subset B^{H,\xi}_{\{t_1,\dots,t_n\}}(y,\epsilon).\end{equation}
    Let $\mathfrak U$ be a collection of elements in $X$ such that \[|\mathfrak U|=\spa(T,\{t_1,\dots,t_n\},\gamma,\delta),\] 
    \[\mu(\bigcup_{y\in\mathfrak U}B^B_{\{t_1,\dots,t_n\}}(y,\gamma))\geq 1-\delta.\]
    Since we proved that $\mu(B_{n,\epsilon})\geq 1-\epsilon$, it follows that
    \[\mu(B_{n,\epsilon}\cap \bigcup_{y\in\mathfrak U}B^B_{\{t_1,\dots,t_n\}}(y,\gamma))\geq 1-\delta-\epsilon.\]
    Then by \Cref{bowen-contained-in-hamming} we also have   
    \[\mu(\bigcup_{y\in\mathfrak U}B^{H,\xi}_{\{t_1,\dots,t_n\}}(y,\epsilon))\geq 1-\delta-\epsilon.\]
    Let us now introduce a convenient notation. If $\alpha$ is a finite measurable partition of $X$, then $N(\alpha,\delta)$ denotes the minimal number of atoms from $\alpha$ needed to cover a subset of $X$ with measure at least $1-\delta$. Since each Hamming pseudo-ball $B^{H,\xi}_{\{t_1,\dots,t_n\}}(y,\epsilon)$ is equal to the union of at most $|B^H_{\xi^n}(\epsilon)|$ atoms of the partition $\xi^{\{t_1,\dots,t_n\}}$, the last inequality shows that
    \[N(\xi^{\{t_1,\dots,t_n\}},\delta+\epsilon)\leq\spa(T,\{t_1,\dots,t_n\},\gamma,\delta) |B^H_{\xi^n}(\epsilon)|.\]
    Let $X_n\subset X$ be a subset with $\mu(X_n)\geq 1-\delta-\epsilon$ and such that $X_n$ is the union of exactly $N(\xi^{\{t_1,\dots,t_n\}},\delta+\epsilon)$ atoms from $\xi^{\{t_1,\dots,t_n\}}$. We now define two Borel measurable partitions $\alpha_n$ and $\beta_n$ of $X$ such that $\xi^{\{t_1,\dots,t_n\}}=\alpha_n\vee\beta_n$. We define $\alpha_n$ as the partition obtained by taking  $\xi^{\{t_1,\dots,t_n\}}$, and collapsing all atoms that do not intersect $X_n$. Similarly, we define $\beta_n$  as the partition obtained by taking  $\xi^{\{t_1,\dots,t_n\}}$, and collapsing the set $X_n$ to a single atom.     By definition of $\alpha_n$, it has at most $\spa(T,\{t_1,\dots,t_n\},\gamma,\delta) |B^H_{\xi^n}(\epsilon)|+1$ atoms. By the trivial bound $H(\alpha_n)\leq \log|\alpha_n|$ and \Cref{estimate-for-cardinality-of-hamming-balls} we have 
    \(\limsup_n \frac{1}{n}H(\alpha_n)\leq \limsup_n \frac{1}{n}\log(\spa(T,\{t_1,\dots,t_n\},\gamma,\delta))+\epsilon \log(|\xi|-1)-\epsilon\log ( \epsilon)-(1-\epsilon)\log(1-\epsilon)\)
    On the other hand, since $\beta_n$ has the atom $X_n$ with measure at least $1-\delta-\epsilon$,  \Cref{partition-with-one-very-big-atom} shows that
    \[
    \limsup_n \frac{1}{n} H(\beta_n)\leq \limsup_n \frac{2/e+(\delta+\epsilon)\log|\beta_n|}{n}.
    \]
    Since $|\beta_n|\leq |\xi^{\{t_1,\dots,t_n\}}|\leq |\xi|^n$, this limsup is at most $(\delta+\epsilon)\log|\xi|$. Since $\xi^{\{t_1,\dots,t_n\}}=\alpha_n\vee \beta_n$ and $H(\xi^{\{t_1,\dots,t_n\}})\leq H(\alpha_n)+H(\beta_n)$, it follows that
    \(
    h^{A}_{\mu}(T,\xi)=
    \limsup_n \frac{1}{n}H(\xi^{\{t_1,\dots,t_n\}})\leq 
    \limsup_n \frac{1}{n}\log(\spa(T,\{t_1,\dots,t_n\},\gamma,\delta))+\epsilon \log(|\xi|-1)-\epsilon\log ( \epsilon)-(1-\epsilon)\log(1-\epsilon)+(\delta+\epsilon)\log|\xi|.\)
    The elements in the last row vanish when $\epsilon$ and $\delta$ tend to $0$. Moreover, this inequality holds for every $\delta\in(0,1)$, $\epsilon\in (0,\frac{|\xi|-1}{|\xi|})$, and $\gamma=\gamma(\epsilon)$ with  $\mu(U_\gamma)<\epsilon^2$.   This  proves  \Cref{katok-v2-inequality-leq}.   
\end{proof}
We are ready to prove part (1) of \Cref{goodman-v2}.
\begin{proof}[Proof of \Cref{goodman-v2}, part (1)]
It is clear that for every $\epsilon>0$ and $\delta\in(0,1)$ we have \[\spa(T,\{t_1,\dots,t_n\},\epsilon,\delta)\leq \spa(T,\{t_1,\dots,t_n\},\epsilon).\]
Therefore by \Cref{katok-v2} and \Cref{sequence-entropy-spanning-sets} we have
\(
h^{A}_{\mu}(T)\leq \lim_{\delta\to^+0}\lim_{\epsilon\to 0^+}\left(\limsup_n \frac{1}{n}\log \spa(T,\{t_1,\dots,t_n\},\epsilon,\delta)\right)\leq\lim_{\epsilon\to 0^+}\left(\limsup_n \frac{1}{n}\log \spa(T,\{t_1,\dots,t_n\},\epsilon)\right) = h^A_{top}(T).\)\end{proof}
As mentioned before, in \Cref{goodman-v2} we derive (2) and (3) from (1). We note that this argument is essentially the same as in Goodman's work \cite{goodman_topological_1974}, and we provide the details for completeness. 
\begin{proof}[Proof of \Cref{goodman-v2}, part (3)]
    Recall that given the sequence $A=\{t_1,\dots\}$, the constant $K(A)$ is defined in 
\Cref{def:K}. We assume that either $K(A)<\infty$ or $h^A_{top}(T)>0$. We must prove 
\begin{equation}\label{topological-krug-newton-2}
    h^A_{top}(T)= \begin{cases}
        K(A)h_{top}(T) &K(A)<\infty, \  h_{top}(T)<\infty\\
        \infty & K(A)=\infty, \ h_{top}(T)>0\\
        \infty & K(A)>0, \ h_{top}(T)=\infty\\
        0&  K(A)=0,  \ h_{top}(T)=\infty.
       \end{cases}
\end{equation}
The inequality $\leq$ in the first and fourth case is proved by Goodman in \cite[Lemma 4.2]{goodman_topological_1974}. Moreover, the second and third cases hold for trivial reasons. Thus we only need to prove the inequality $\geq$. 

Lemma 4.3 in \cite{goodman_topological_1974}	states that for every $\mu\in M_T(X)$ we have
\begin{equation}\label{krug-newton-2}
		h^{A}_{\mu}(T) \geq 
		\begin{cases}
			K(A)h_{\mu}(T) &K(A)<\infty, \  h_{\mu}(T)<\infty\\
			\infty & K(A)=\infty, \  h_{\mu}(T)>0\\
			\infty & K(A)>0, \  h_{\mu}(T)=\infty\\
			0&  K(A)=0,  \  h_{\mu}(T)=\infty.
		\end{cases}
\end{equation}
In order to prove \Cref{topological-krug-newton-2}, let us assume first that we are in the case ($K(A)=\infty$ and $h_{top}(T)<\infty$). We have: 
\(
h^A_{top}(T)\geq \sup _{\mu\in M_T(X)} h^{A}_{\mu}(T)
\geq \sup _{\mu\in M_T(X)}K(A)h_\mu(T) \geq K(A)\sup _{\mu\in M_T(X)}h_\mu(T)=K(A)h_{top}(T)
\)
The first inequality is by (1) in \Cref{goodman-v2}, which we already proved. The second inequality is \Cref{krug-newton-2}. Since we assumed $h_{top}(T)<\infty$, the standard variational principle shows that $h_\mu(T)<\infty$ for every $\mu\in M_T(X)$.  Thus \Cref{krug-newton-2} is used in the first case, which says  $h^{A}_{\mu}(T)\geq K(A)h_{\mu}(T)$. The third inequality is clear, and the last inequality is the standard variational principle. Thus we have proved the claim in the case ($K(A)=\infty$ and $h_{top}(T)<\infty$). The remaining cases can be proved with a similar argument.
\end{proof}
\begin{proof}[Proof of \Cref{goodman-v2}, part (2)]
We must prove that unless ($K(A)=\infty$ and $h_{top}(T)=0$), the variational principle holds for sequence entropy:
\begin{equation}\label{sequential-variational-principle-proof}
    \sup_{\mu\in M_T(X)} h^{A}_{\mu}(T) = h^A_{top}(T)
\end{equation}
We already proved the inequality $\leq$, so we just need to focus on the inequality $\geq$. As before, we will consider cases according to the values of $K(A)$ and $h_{top}(T)$.

Let us first consider the case ($h_{top}(T)<\infty$ and $h_{top}(T)<\infty$). We have:
\(
\sup_{\mu\in M_T(X)}h^{A}_{\mu}(T)\geq \sup _{\mu\in M_T(X)} K(A)h_{\mu}(T)=K(A)\sup _{\mu\in M_T(X)} h_{\mu}(T)=K(A)h_{top}(T)=h^A_{top}(T)
\)
The first inequality is by \Cref{krug-newton-2}. Observe that by our assumption $h_{top}(T)<\infty$ and by the standard variational principle, we have $h_\mu(T)<\infty$ for every  $\mu\in M_T(X)$. Therefore we are applying the first case of \Cref{krug-newton-2}, which says $h^{A}_{\mu}(T)\geq K(A)h_{\mu}(T)$. The third equality is by the standard variational principle, and the last equality is by part (3) of \Cref{goodman-v2}, which we already proved. 

Thus we have verified the inequality $\geq$ for \Cref{sequential-variational-principle-proof} in the case ($h_{top}(T)<\infty$ and $h_{top}(T)<\infty$). The remaining cases ($K(A)=\infty$ and $h_{top}(T)>0$), ($K(A)>0$ and $h_{top}(T)=\infty$) and ($K(A)=0$ and $h_{top}(T)=\infty$) follow by the same argument. 
\end{proof}
\begin{remark}\label{remark:explanation-goodman-proof}
Goodman's proof of part (1) in \Cref{goodman-v2} is similar in ideas to the well-known proof of the variational principle of Misiurewicz \cite{misiurewicz_short_nodate}, see also \cite{petersen_ergodic_1983,walters_introduction_2000}. That is, by approximating a measurable partition by open (or closed) covers one first establishes a relation of the form $h_{\mu}(T)\leq h_{top}(T)+C$, for some constant $C>0$. Then one replaces $T$ by its powers  $T^n$ or products $T\times\dots\times T$ to remove the constant $C$. The second step is not possible for sequence entropy, which may not behave like entropy with respect to powers and products \cite{lemanczyk_sequence_1985,hulse_counterexamples_2009}. We also note that the argument using Hamming (pseudo)-balls can be  generalized to prove that measure-theoretic entropy is at most topological entropy in the context of amenable group actions (compare with \cite[Theorem 9.48]{kerr_ergodic_2016}).
\end{remark}
\begin{remark}
    In the next section we will prove  estimates for $\spa(T,F,\epsilon)$ for an invertible transformation $T$ and for a set $F\Subset\Z$ (\Cref{estimate}). These can be interpreted as a finitary version of part (3) in \Cref{goodman-v2}. Indeed, an alternative proof of part (3) in \Cref{goodman-v2} can be obtained by taking $F=\{t_1,\dots,t_n\}$ in \Cref{estimate}, item (2). 
\end{remark}
\section{Proof of \Cref{thm:entropy} and \Cref{thm:slow-entropy}}\label{sec:slow-entropy}
In this section we prove our results regarding topological slow entropy of skew products, \Cref{thm:entropy} and \Cref{thm:slow-entropy}. We recall that notation and terminology for Bowen metrics and subshifts are reviewed in \Cref{sec:preliminaries}.

\subsection{Estimating  $\spa(T,F,\epsilon)$}\label{sec:estimates}
In this subsection we present some preliminary computations that will be needed for the proof of \Cref{thm:entropy}.

Let $T$ be a continuous and invertible transformation of the compact metric space $(X,d)$. Here we prove estimates for $\spa(T,F,\epsilon)$, where $F\Subset \Z$ is not necessarily an interval, but its ``gaps'' are bounded by a constant $m$. 

\begin{definition}\label{constant-C_m} Let $m\in\N$. A set $F\Subset \Z$ is said to have $m$-\textbf{bounded gaps} if the set $F+\{0,\dots,m-1\}$ is an interval. We define $C_m(F)\in\R $ by $
C_m(F)=|F+\{0,\dots,m-1\}|/|F|$.
\end{definition} The main result of this subsection is the following.
\begin{proposition}\label{estimate}
Fix $m\in\N$. Suppose that $h_{top}(T)<\infty$. Then for every  $\epsilon>0$ we can choose $\delta\in (0,\epsilon)$ and $n_0\in\N$ such that
\begin{enumerate}
    \item $\frac{1}{|F|}\log \spa(T,F,\delta)\leq C_m(F) (h_{top}(T)+\epsilon)$ for every $F\Subset\Z$ with $m$-bounded gaps and $|F|\geq n_0$,
    \item $
    \frac{1}{|F|}\log \spa(T,F,\delta)\geq C_m(F) (h_{top}(T)-\epsilon)$ for every $F\Subset\Z$.
\end{enumerate}
    If $h_{top}(T)=\infty$ then for every $N\in\N$ we can choose $\delta>0$ such that $C_m(F)\cdot N\leq \frac{1}{|F|}\log \spa(T,F,\delta)$ for every $F\Subset\Z$.
\end{proposition} 
One easily derives an analogous statement in terms of open covers using \Cref{wellknown}. For the proof we we will use the following result of Downarowicz, Frej, and  Romagnoli, which is known as the infimum rule. 
\begin{proposition}[Theorem 6.8 in \cite{kolyada_shearers_2016}]\label{downarowicz}
The topological entropy $h_{top}(T)$ satisfies
\[\label{eq:downarowicz}
        h_{top}(T)=\sup_{\mathcal U}\inf_{F\Subset \Z}\frac{1}{|F|}\log(N(\mathcal{U}^F)).
\]
    Here $\U$ ranges over finite open covers for $X$.
\end{proposition}
\begin{proof}[Proof of \Cref{estimate}]
    Let $\epsilon>0$. We start by proving the statement in the case $h_{top}(T)<\infty$. We start by choosing $\delta_1$ so that for every $\delta<\delta_1$ we have
\[h_{top}(T,\delta)=\limsup_{n\to\infty} \frac{1}{n} \log \spa(T,\{0,\dots,n-1\},\delta)<h_{top}(T)+\epsilon.
\]
This is possible as $h_{top}(T)=\lim_{\delta\to 0^+}h_{top}(T,\delta)$, also recall that we defined $h_{top}(T,\delta) $  in \Cref{definition-h-T-epsilon}. By definition of limsup, it follows that there is $n_0$ such that for all $n\geq n_0$ we have 
\[ 
 \frac{1}{n}\log \spa(T,\{0,\dots,n-1\},\delta)<h_{top}(T)+\epsilon.\]
\Cref{monotony-properties-of-spa} shows that then for every interval $I\Subset \Z$ with $|I|\geq n_0$ we also have 
\begin{dmath}\frac{1}{|I|}\log \spa(T,I,\delta)<h_{top}(T)+\epsilon.\end{dmath}
Now let $F\Subset \Z$ with $m$-bounded gaps and with $|F|\geq n_0$, and let $I=F+\{0,\dots,m-1\}$. The set $I$ is an interval because $F$ has $m$-bounded gaps. Moreover the  inequality above is valid for $I$ as $|I|\geq n_0$. Since  $\spa(T,F,\delta)\leq\spa(T,I,\delta)$ (\Cref{monotony-properties-of-spa}) and $1/|F|=C_m(F)/|I|$, we have
\(
\frac{1}{|F|}\log \spa(T,F,\delta)\leq  \frac{1}{|F|}\log \spa(T,I,\delta) =  \frac{C_m(F)}{|I|}\log \spa(T,I,\delta)\leq C_m(F)\cdot(h_{top}(T)+\epsilon).
\)
This shows that $\delta$ verifies part (1) in 
the statement. 
We now prove part (2). By \Cref{downarowicz} we can choose a finite open cover $\mathcal V$ of $X$ such that
\[\inf_{F\Subset \Z}\frac{\log N(\mathcal V ^F)}{|F|}> h_{top}(T)-\epsilon.\]
Now consider the open cover $\U=\mathcal V^{\{0,\dots,m-1\}}$. For any $F\Subset \Z$ we have $\U^F=\mathcal V^{F+\{0,\dots,m-1\}}$. Using the fact $1/|F|=C_m(F)/|F+\{0,\dots,m-1\}|$ and the property of $\mathcal V$ we have
\begin{dmath}\frac{\log N(\U^F)}{|F|}=\frac{\log N(\mathcal V^{\{0,\dots,m-1\}+F})}{|F|}=C_m(F)\frac{\log N(\mathcal V^{\{0,\dots,m-1\}+F})}{|F+\{0,\dots,m-1\}|} \geq C_m(F)\cdot(h_{top}(T)-\epsilon)\end{dmath}
Thus $\U$ has the property $\frac{1}{|F|}\log N(\U^F)\geq C_m(F)\cdot (h_{top}(T)-\epsilon)$, $F\Subset \Z$. Now let $\delta_2>0$ be smaller than one half the Lebesgue number for $\U$, and let $\delta<\delta_2$. By \Cref{wellknown} and \Cref{sep-and-spa} we have  $\spa(T,F,\delta)\geq \sep(T,F,2\delta)\geq N(\U^F)$, and therefore 
\(\label{cool-inequality-2}
\frac{\log \spa(T,F,\delta)}{|F|}\geq  \frac{\log N(\U^F)}{|F|}\geq  C_m(F)\cdot (h_{top}(T)-\epsilon), \ \ F\Subset\Z.
\)
Thus $\delta$ verifies part (2) in the claim. The conclusion in the statement holds for every $\delta<\min\{\delta_1,\delta_2\}$. 

We now consider the case $h_{top}(T)=\infty$. \Cref{downarowicz} shows that given $N\in\N$ we can choose a finite open cover $\mathcal V$ such that $\log N(\mathcal V^F)\geq |F|\cdot N$ for every $F$. As before, we let $\delta$ be smaller than the Lebesgue number of the open cover $\mathcal V^{\{0,\dots,m-1\}}$ multiplied by $1/2$. The same computations from before show that the inequality in the statement is verified by $\delta$. 
\end{proof}

\begin{remark} Intuitively speaking, 
\Cref{estimate} says that for a finite entropy system,  within the class of sets  $F\Subset\Z$ with $m$-bounded gaps  we can estimate \[\spa(T,F,\epsilon)\approx e^{C_m(F)\cdot |F|\cdot h}.\]
For instance if we take $F_n=\{1,\dots,n\}$ then $C_m(F_n)$ converges to $1$ and we recover the familiar estimate $\spa(T,F_n,\epsilon)\approx e^{|F_n|h}= e^{nh}$. For $F_n=\{2,4,6,8\dots,2n\}$ we have that $C_m(F_n)$ converges to $2$ and therefore  $\spa(T,F_n,\epsilon)\approx e^{2|F_n|h}=e^{2nh}$. Observe that $C_m(F)$ behaves as a finitary version of the constant $K(A)$ from \Cref{Topological-kruw-newton} for a sequence $A=(t_i)_{i\geq 1}$. In fact, provided that $A$ has no repetitions, $K(A)$ is the limit of $C_m(\cdot)$ evaluated over finite segments of the sequence, in the sense that \[K(A)=\lim_{m\to\infty}\limsup_{n\to\infty}C_m(\{t_1,\dots,t_n\}).\]
\end{remark}

We finish this section stating an interesting corollary of \Cref{estimate}. This will not be used in the rest of this work and is related to the questions considered  in \cite{kolyada_shearers_2016}. It is well known that topological entropy can be computed as $h_{top}(T)=\sup_{\U}\limsup_{n\to\infty}\frac{1}{|F_n|}\log N(\U^{F_n})$  for every choice of a F\o lner sequence $(F_n)_{n\in\N}$,  see \cite{kolyada_shearers_2016}. A nontrivial consequence of \Cref{downarowicz} 
is that for every choice of  $(F_n)_{n\in\N}$, even if it is not a F\o lner sequence, we still have the inequality $h_{top}(T)\leq \sup_{\U}\limsup_{n\to\infty}\frac{1}{|F_n|}\log N(\U^{F_n})$. The next result  complements these facts by showing that the inequality is strict for every choice of a non-F\o lner sequence. 
\begin{corollary}[Of \Cref{estimate}]
Suppose that $0<h_{top}(T)<\infty$ and let $(F_n)_{n\in\N}$ be an arbitrary sequence of finite subsets of $\Z$ which is not a F\o lner sequence. Then 
\begin{equation}\label{limsup-nonfolner}h_{top}(T)<\sup_{\U}\limsup_{n\to \infty}\frac{1}{|F_n|}\log N(\U^{F_n}).\end{equation}
\end{corollary}
\begin{proof}
Recall that $(F_n)_{n\in\N}$, $F_n\Subset \Z$, is a F\o lner sequence if for every $E\Subset\Z$ the quotient $|(F_n+E)\Delta F_n|/|F_n|$ converges to $0$ as $n\to\infty$. Here $\Delta$ denotes symmetric difference. Furthermore, it suffices to consider sets $E$ of the form $\{0,\dots,m-1\}$. See  \cite[Proposition 4.6.1]{ceccherini-silberstein_cellular_2010} for more details. Since $|F_n+\{0,\dots,m-1\}|=|F_n|+|(F_n+\{0,\dots,m-1\})\Delta F_n|$, it follows that  we can write
\[C_m(F_n)=1+\frac{|(F_n+\{0,\dots,m-1\})\Delta F_n|}{|F_n|}.\]
This shows that the sequence is F\o lner if and only if $\lim_{n\to\infty}C_m(F_n)=1$ for every $m\in\N$. 

Now let $(F_n)_{n\in\N}$ be a sequence of finite subsets of $\Z$ which is not a F\o lner sequence. Note that by definition $C_m(F_n)$ is always at least $1$. By the observation in the previous paragraph, we have that for $m\in\N$ is large enough $\limsup_{n\to\infty}C_m(F_n)$ is strictly larger than 1. We assume in what follows that $m$ has this property and let $L=\limsup_{n\to\infty}C_m(F_n)$.  

Let $\epsilon>0$ be arbitrary, and choose $\delta\in(0,\epsilon)$ as in \Cref{estimate}. Thus for every $n\in\N$ we have \[\frac{1}{|F_n|}\log \spa(T,F_n,\delta)\geq 
 C_m(F_n)(h_{top}(T) -\epsilon)\]
 Now let $\U$ be a finite open cover of $X$ whose elements have diameter smaller than $\delta$. Since $N(\U^{F_n})\geq \spa(T,F_n,\delta)$ for all $n$ (\Cref{wellknown}), it follows that 
 \( \limsup_{n\to\infty}\frac{1}{|F_n|}\log N(\U^{F_n}) \geq \limsup_n \ C_m(F_n)(h_{top}(T) -\epsilon) = L\cdot (h_{top}(T)-\epsilon)
 \)
 Thus the supremum in the statement is at least $L\cdot (h_{top}(T)-\epsilon)$. Since this is true for arbitrary $\epsilon$, it follows that it is at least  $L\cdot h_{top}(T)$.  
\end{proof}
Note that \Cref{limsup-nonfolner} may not hold
if we replace $\limsup$ by $\liminf$ or $\inf$. This occurs exactly when a subsequence of $(F_n)$ is F\o lner. In fact, it is possible to estimate the gap between the limsup and liminf by considering  the limsup and liminf of $C_m(F_n)$ as $n\to\infty$, and following the same idea in the previous proof. 

\subsection{Sketch of the proof of \Cref{thm:entropy}}
Here we present an informal outline of the proof of \Cref{thm:entropy}. 

For simplicity we first describe the proof in a special case first.  Let $Y\subset \{-1,1\}^\Z$ be a subshift on symbols $\{-1,1\}$, let $S$ be the shift map on $Y$, and let $\tau\colon Y\to\Z$ be the continuous map defined by $\tau(y)=y(0)$, $y\in Y$. We also fix an arbitrary invertible topological system $(X,T)$, and  consider the map $S\rtimes_\tau T$ over $ Y\times X$ defined by $(y,x)\mapsto (S(y), T^{\tau(y)}(x))$.

\Cref{thm:entropy} states that there is a scale $\{a_n(t)\}_{n\in\N,t>0}$  depending only on $(Y,S,\tau)$, and with the property $\upperent_{\mathbf{a}}(S\rtimes_\tau T) =h_{top}(T)$. We will prove that the following scale has the desired property:
\begin{equation}\label{first-scale}a_n(t)=\sum_{w\in L_n(Y)} e^{t\cdot |\tau^{\{0,\dots,n-1\}}(w)|}, \ \ \  n\in\N, t>0.\end{equation}
Here $L_n(Y)$ is the set of words of length $n$ that appear in $Y$. For each $w\in L_n(Y)$ we define $\tau^{\{0,\dots,n-1\}}(w)$ as the set $\{\tau^i(w) : i=0,\dots,n-1\}\Subset\Z$, where $\tau^i(w)=w_0+w_1+\dots+w_i$ for $i\geq 1$, and $\tau^0(w)=0$. If we interpret $w$ as a walk on $\Z$ of $n$ steps, where each $-1$ or $1$ corresponds to ``left'' and ``right'', then $|\tau^{\{0,\dots,n-1\}}(w)|$ is the number of visited places in $\Z$.

Recall that topological slow entropy is defined in  \Cref{sec:preliminaries}. In order to compute the topological slow entropy of $S\rtimes_\tau T$ with scale $a_n(t)$ we must estimate $\spa(S\rtimes_\tau T,\{0,\dots,n-1\},\epsilon)$ and then compare this with $a_n(t)$. In \Cref{prop:the-set-A} we will show an estimate of the form 
\begin{equation}\label{informal-estimate-for-spa}
    \spa(S\rtimes_\tau T,\{0,\dots,n-1\},\epsilon)\approx\sum_{w\in L_n(Y)}\spa(T,\tau^{\{0,\dots,n-1\}}(w),\epsilon).
\end{equation}
An important consequence of the fact that $\tau$ attains only values in $\{-1,1\}$ is that the set $\tau^{\{0,\dots,n-1\}}(w)$ is an interval for all $n\in\N$, $w\in L_n(Y)$. The definition of topological entropy shows that $\spa(T,I,\epsilon)\approx e^{h_{top}(T)|I|}$ for any interval $I\Subset\Z$. Thus we can estimate
\[\spa(S\rtimes_\tau T,\{0,\dots,n-1\},\epsilon)\approx\sum_{w\in L_n(Y)}e^{h_{top}(T)|\tau^{\{0,\dots,n-1\}}(w)|}.\]
Therefore for $t>0$ we can write
\begin{equation}\label{quotient-to-analyze}
\frac{\spa(S\rtimes T,\{0,\dots,n-1\},\epsilon)}{a_n(t)}\approx \frac{\sum_{w\in L_n(Y)}e^{h_{top}(T)|\tau^{\{0,\dots,n-1\}}(w)|}}{\sum_{w\in L_n(Y)} e^{t|\tau^{\{0,\dots,n-1\}}(w)|}}
\end{equation}
The assumption that we impose on $\tau$ essentially means that $|\tau^{\{0,\dots,n-1\}}(w)|$ is unbounded on a positive proportion of $w\in L_{n}(Y)$, as $n$ grows. This allows us to prove that the limsup of the quotient above is $0$ if $t>h_{top}(T)$, and that the liminf of this quotient is positive for $t<h_{top}(T)$. Thus the upper topological slow entropy of $S\rtimes_{\tau}T$ with scale $a_n(t)$ is at most $h_{top}(T)$, and the lower topological slow entropy of $S\rtimes_{\tau}T$ with scale $a_n(t)$ is at least $h_{top}(T)$. This implies that both quantities are equal to $h_{top}(T)$, which is exactly what we wanted.

Following these same lines we will prove \Cref{thm:entropy} in the general case. The biggest challenge is that $\tau$ may attain values outside $\{-1,0,1\}$. The sets $\tau^{\{0,\dots,n-1\}}(w)$ may not be intervals, and the sparseness of these sets may produce extra growth which is quantified by the constant $C_m$ from \Cref{sec:estimates} (for suitable $m$). This makes it necessary to modify the scale from \Cref{first-scale}, and the scale having the desired property in the general case is
\[a_n(t)=\sum_{w\in L_n(Y)} e^{t\cdot |\tau^{\{0,\dots,n-1\}}(w)|\cdot C_m(\tau^{\{0,\dots,n-1\}}(w))}, \ \ \  n\in\N, t>0.\]

\subsection{The tuple $(Y,S,\tau)$ and the property of being $\lambda$-unbounded}\label{subsec:the-tuple-S}
Here we fix the tuple to   $(Y,S,\tau)$ with which we will work in the proof of \Cref{thm:entropy} and \Cref{thm:slow-entropy}. We also define the property of being $\lambda$-unbounded that we impose on $\tau$. We will assume that the subshift in the base is a $\Z$-subshift, but the proof can be easily adapted to the case of an $\N$-subshift, see \Cref{non-invertible}.  
 
Let $Y\subset \A^\Z$ be a subshift on alphabet $\A$, $|\A
\geq 2$, and let $S$ be the shift map on $Y$. 
We endow $Y$ with the standard metric for subshifts, so the distance between two configurations is $2^{-n}$ when they coincide over $\{-n,\dots,n\}$ and they do not over $\{-n-1,\dots,n+1\}$.

Let $\tau\colon Y \to \Z$ be a continuous function.  We now observe some consequences of this assumption. Since $Y$ is compact and $\tau$ is continuous, it follows that $\tau (Y)$ is a compact subset of $\Z$ and therefore finite. Thus $\tau$ attains finitely many values and we can  fix $m\in\N$ with the property 
\begin{equation}\label{constant-m}|\tau(y)|\leq m \ \text{ for all } y\in Y.\end{equation}
The results from \Cref{sec:estimates} will be used for this  fixed $m$ (the $m$ in \Cref{constant-C_m}). 

We can write a disjoint union  $Y=\bigsqcup _{i=-m}^{m}\tau^{-1}\{i\}$. Since $\tau$ is continuous and $\Z$ is given the discrete topology, each $\tau^{-1}\{i\}$ is both open and closed as a subset of $Y$. One easily sees that any subset of $Y$ which is both open and closed in a subshift is a finite union of cylinders\footnote{If  $C\subset Y$ is open then it equals a union of balls $C=\bigcup_{i\in I} U_i$. In this metric a ball is a cylinder, so each $U_i$ is a cylinder. Now suppose that $C$ is also closed. Since $Y$ is compact we have that $C$ is compact, $\{U_i\}_{i\in I}$ is an open cover of $C$, and taking a finite subcover $U_1,\dots, U_n$ we obtain $C=\bigcup _{i=1}^n U_i$.}. Thus we can find  $s\in\N$ with the property
\begin{equation}\label{constant-s}
    d(y,y')\leq 2^{-s}\Rightarrow \tau(y)=\tau(y'), \ \text{for all }y,y'\in Y.
\end{equation}
In other words for all $y\in Y$ the value  $\tau(y)$ is determined by the values of $y$ over $\{-s,\dots,s\}$.

Given $y\in Y$ we define $\tau^0(y)=0$,  and for $n\geq 1$ 
\[\tau^n(y) = \sum_{i=0}^{n-1}\tau(S^i(y)). \]
Furthermore, if $F\Subset \Z$ then we define $\tau^F(y)$ as the set
\[\tau^F(y)=\{\tau^i(y) : i\in F\}\Subset\Z.\]   
In some cases it makes sense to evaluate $\tau$ on a pattern or a word instead of a configuration, as $\tau(y)$ depends on finitely many values of $y$. Recall that $L_{n,s}(Y)$ denotes $\{y|_{\{-s,\dots,s+n-1\}} : y\in Y\}$. Given $w\in L_{n,s}(Y)$ we can write  $\tau^i(w)$ for every $i\in \{0,\dots,n-1\}$, and then also $\tau^{\{0,\dots,n-1\}}(w)$. 

We remark that despite the set $\tau^{\{0,\dots,n-1\}}(w)$ may not be an interval, it must have $m$-bounded gaps (by \Cref{constant-m}). Therefore we can apply \Cref{estimate} to sets of the form $\tau^{\{0,\dots,n-1\}}(w)$.

We are now ready to define the assumption on $\lambda$ that we will need.
\begin{definition}\label{def:good}\label{lambda-unbounded}
    Given $w\in L_{n,s}(Y)$ we define $r_n(w)$ by 
    \[r_n(w)=|\tau^{\{0,\dots,n-1\}}(w)|.\]
    Given $\lambda>0$, we say that $\tau$ is $\lambda$-unbounded if for
    all $N\in\N$ we have
\[\liminf_{n\in\N}\frac{|\{w\in L_{n,s}(Y) : r_n(w)\geq N\}|}{|L_{n,s}(Y)|}\geq \lambda.\]
\end{definition}
We assume in what follows that $\tau$ has this property for some $\lambda>0$. One easily verifies that this property does not depend on the choice of $s$ from \Cref{constant-s}. In \Cref{sec:examples} we will observe  sufficient conditions that guarantee that $\tau$ is $\lambda$-unbounded. 

Finally, observe that if $\tau$ is $\lambda$-unbounded then it is also $\lambda'$-unbounded for any $0<\lambda'<\lambda$. It will be convenient for us to fix $\ell\in\N$ such that $\tau$ is $1/\ell$-unbounded. 

\subsection{The scales}
Here we define the scales that will be used to prove \Cref{thm:entropy} and \Cref{thm:slow-entropy}. 
\begin{definition}\label{def:scales}
Let $w\in L_{n,s}(Y)$, and recall from \Cref{def:good} that
\[r_n(w)= |\tau^{\{0,\dots,n-1\}}(w)|.\]
We define $q_n(w)$ by 
\[q_n(w)=r_n(w)\cdot C_m(\tau^{\{0,\dots,n-1\}}(w))\]
Here $C_m$ is the constant defined in \Cref{constant-C_m}, and $m$ was chosen in \Cref{constant-m}. We define the scale $a_n(t)$ by
\[a_n(t)=\sum_{w\in L_{n,s}(Y)}e^{t\cdot q_n(w)}.\]
Furthermore, given an arbitrary scale $b_n(t)$, we define the scale $c_n(t)$ by 
\[c_n(t)=\sum_{w\in L_{n,s}(Y)}b_{r_n(w)}(t).\]
\end{definition}
We will prove that $a_n(t)$ verifies the conclusion in \Cref{thm:entropy}. Assuming that $\tau$ attains only values in $\{-1,0,1\}$, we will show  that $c_n(t)$ verifies the conclusion in \Cref{thm:slow-entropy}.
\subsection{The skew product $S\rtimes_\tau T$}
Let $X$ be a compact metric space and let $T\colon X\to X$ be a homeomorphism. Here we define the skew product $S\rtimes_\tau T$ and prove some elementary results. 

We endow the product space $Y\times X$ with the metric $d^{Y\times X}((y,x), (y',x'))=\max\{d^Y(y,y'),d^X(x,x')\}$. Here $d^X$ denotes the metric that comes with $X$, and $d^Y$ denotes the standard metric for subshifts (\Cref{sec:preliminaries}). We will omit the superscripts as the spaces ($X$, $Y$, or $Y\times X$) will be clear from the context. Consider the continuous transformation $S\rtimes_\tau T$ over $Y\times X$ defined by 
\[(y,x)\mapsto (S(y),T^{\tau(y)}(x)).\]
The $n$-th iterate of this map is given by 
\[(y,x)\mapsto(S\rtimes_\tau T)^n(y,x)=(S^n(y),T^{\tau^n(y)}(x)).\]
Let $(y,x)$ and $(y',x')$ be a pair of elements in $Y\times X$ such that they are $\epsilon,\{0,\dots,n-1\}$-close in the Bowen metric for $S\rtimes_\tau T$, for some $\epsilon\leq 2^{-s}$. This implies that $y$ and $y'$ are $\epsilon,\{0,\dots,n-1\}$-close in the Bowen metric for $S$. Since we took $\epsilon$ sufficiently small, this implies that they coincide over $\{-s,\dots,n-1+s\}$, and therefore $\tau^i(y)=\tau^i(y')$ for every $i\in\{0,\dots,n-1\}$. Thus $d^B_{\{0,\dots,n-1\}}((y,x),(y',x'))$ is equal to
\[\max(\{d(S^i(y),S^i(y')) : i\in \{0,\dots,n\}\}\cup \{d(T^i(x),T^i(x')) : i \in \tau^{\{0,\dots,n-1\}}(y)\}).\]
The relevant conclusion is that whenever $d^B_{\{0,\dots,n-1\}}((y,x),(y',x'))$ is smaller than $2^{-s}$, then the Bowen distance for $S\rtimes_\tau T$ can be expressed in terms of the Bowen distances for $S$ and $T$ as
\begin{equation}\label{lemma-bowen-metrics} d^B_{\{0,\dots,n-1\}}((y,x),(y',x'))=\max\{d^B_{\{0,\dots,n-1\}}(y,y'),d^B_{\tau^{\{0,\dots,n-1\}}(y)}(x,x')\}\end{equation}
We will now estimate  $\spa(S\rtimes_\tau T,\{0,\dots,n-1\},\epsilon)$. It will be convenient to introduce the following. Recall that $L_{n,s}(Y)$ is the set $\{y|_{\{-s,\dots,n+s\}}:y\in Y\}$.
\begin{definition}
    For all $\epsilon\in (0,2^{-s})$ and $n\geq 1$,  we define $A_n(\epsilon)$  by
    \[A_n(\epsilon)=\sum_{w\in L_{n,s}(Y)}\spa(T,\tau^{\{0,\dots,n-1\}}(w),\epsilon).\]
\end{definition}
The next result shows that in the computation of the topological slow entropy of $S\rtimes_\tau T$, we can replace $\spa(S\rtimes_\tau T,\{0,\dots,n-1\},\epsilon)$ by $A_n(\epsilon)$, which is easier to analyze. 
\begin{proposition}\label{prop:the-set-A}
    Let $\epsilon\in (0,2^{-s-1})$ and $n\geq 1$. Then we have the inequalities
    \[A_n(2\epsilon)\leq \spa(S\rtimes_\tau T,\{0,\dots,n-1\},\epsilon)\leq E_{\epsilon}\cdot  A_n(\epsilon/2),\]
    where $E_{\epsilon}$ is a constant depending on $\epsilon$ and not on $n$.
\end{proposition}
\begin{proof}
For each $\epsilon\in(0,2^{-s})$ and $F\Subset \Z$ we let $C(T,F,\epsilon)$ be a choice of a maximal $\epsilon,F$-separated set for $T$. Similarly, we denote by $C(S,F,\epsilon)$ a choice of a maximal $\epsilon,F$-separated set for $S$. It follows that these sets are also $\epsilon,F$-spanning. For all $\epsilon\in(0,2^{-s})$ and $n\geq 1$, we define 
\[B(n,\epsilon)=\{(y,x) : y\in C(S,\{0,\dots,n-1\},\epsilon), \ x\in C(T,\tau^{\{0,\dots,n-1\}}(y),\epsilon)\}.\]
We claim that $B(n,\epsilon)$ is $\epsilon,\{0,\dots,n-1\}$-spanning and $\epsilon,\{0,\dots,n-1\}$-separated for $S\rtimes_\tau T$. It is straightforward that $B(n,\epsilon)$ is $\epsilon,\{0,\dots,n-1\}$-spanning. To see that it is $\epsilon,\{0,\dots,n-1\}$-separated, let $(y,x)$ and $(y',x')$ be elements in $B(n,\epsilon)$ that are $\epsilon,\{0,\dots,n-1\}$-close.  \Cref{lemma-bowen-metrics} shows that then $x$ and $x'$ are $\epsilon, \tau^{\{0,\dots,n-1\}}(y)$-close. Since $x$ and $x'$ belong to the $\tau^{\{0,\dots,n-1\}}(y),\epsilon$-separated set $C(T,\tau^{\{0,\dots,n-1\}}(y),\epsilon)$, this shows that $x=x'$. The same argument applies to $y$ and $y'$ and therefore $(y,x)=(y',x')$.

We can now prove the left inequality in the statement. Observe that for every $w\in L_{n,s}(Y)$ it is possible to find $y\in C(S,\{0,\dots,n-1\},\epsilon)$ such that $w\sqsubset y$  and in this case $\tau^{\{0,\dots,n-1\}}(w)=\tau^{\{0,\dots,n-1\}}(y)$. Thus we have
\(
A_n(\epsilon)=\sum_{w\in L_{n,s}(Y)}\spa(T,\tau^{\{0,\dots,n-1\}}(w),\epsilon) \leq 
\sum_{y\in C(S,\{0,\dots,n-1\},\epsilon)} \spa(T,\tau^{\{0,\dots,n-1\}}(y),\epsilon)
\)
Since $C(T,\tau^{\{0,\dots,n-1\}}(y),\epsilon)$ is in particular $\epsilon,\{0,\dots,n-1\}$-spanning we have that $\spa(T,\tau^{\{0,\dots,n-1\}}(y),\epsilon)\leq \lvert C(T,\tau^{\{0,\dots,n-1\}}(y),\epsilon)\rvert$ for each $y$, and the inequality above becomes 
\(  
\phantom{a} \leq 
\sum_{y\in C(S,\{0,\dots,n-1\},\epsilon)} \lvert C(T,\tau^{\{0,\dots,n-1\}}(y),\epsilon)\rvert
= \lvert B(n,\epsilon)\rvert 
\)
Since  $B(n,\epsilon)$ is in particular $\epsilon,\{0,\dots,n-1\}$-separated, its cardinality can be bounded as
\(\phantom{a}\leq \sep 
 (S\rtimes_\tau T,\{0,\dots,n-1\},\epsilon) \leq\spa(S\rtimes_\tau T,\{0,\dots,n-1\},\epsilon/2).\)
Here we used \Cref{sep-and-spa}. Thus $A_n(\epsilon)\leq \spa(S\rtimes_\tau T,\{0,\dots,n-1\},\epsilon/2)$ and this proves the left inequality in the statement.

We now prove the remaining inequality in the statement $\spa(S\rtimes_\tau T,\{0,\dots,n-1\},\epsilon)\leq E_{\epsilon}A_n(\epsilon/2)$. Since $B(n,\epsilon)$ is $\epsilon,\{0,\dots,n-1\}$-spanning, we have  $\spa(S\rtimes_\tau T,\{0,\dots,n-1\},\epsilon)\leq |B(n,\epsilon)|$. Thus it suffices to prove 
\begin{equation}\label{equation-c-epsilon}|B(n,\epsilon)|\leq E_{\epsilon}A_n(\epsilon/2),\end{equation} for a constant $E_{\epsilon}$ independent of $n$. Using \Cref{sep-and-spa} and the fact that $C(T,\tau^{\{0,\dots,n-1\}}(y),\epsilon)$ is in particular $
\epsilon,\{0,\dots,n-1\}$-separated we have
\begin{dmath}\label{eq:long sum}
\vert B(n,\epsilon)\rvert = \sum_{y\in C(S,\{0,\dots,n-1\},\epsilon)}\lvert C(T,\tau^{\{0,\dots,n-1\}}(y),\epsilon)\rvert \leq{\sum_{y\in C(S,\{0,\dots,n-1\},\epsilon)}\sep(T,\tau^{\{0,\dots,n-1\}}(y),\epsilon)}\leq\sum_{y\in C(S,\{0,\dots,n-1\},\epsilon)}\spa(T,\tau^{\{0,\dots,n-1\}}(y),\epsilon/2)
\end{dmath}
Let $l(\epsilon)=\lceil\log_2(\epsilon)\rceil$, so that $d(y,y')\leq\epsilon$ exactly when $y,y'$ coincide when restricted to $\{-l(\epsilon),\dots,l(\epsilon)\}$. It follows that $d^B_{\{0,\dots,n-1\}}(y,y')\leq\epsilon$ exactly when $y,y'$ coincide over  $\{-l(\epsilon),\dots,n+l(\epsilon)-1\}$. Since $C(S,\{0,\dots,n-1\},\epsilon)$ is chosen to be $\epsilon,\{0,\dots,n-1\}$-separated, every element in this set is uniquely identified by its restriction to $\{-l(\epsilon),\dots,n+l(\epsilon)-1\}$. Thus for each $w\in L_{\{-s,\dots,n+s-1\}}(Y)$ there are at most $E_\epsilon =|\A|^{n+2l(\epsilon)-2s}$ elements $y\in C(S,\{0,\dots,n-1\},\epsilon)$ with $y|_{\{-s,\dots,s+n-1\}}=w$. If $w$ and $y$ satisfy this relation then we have $\tau^{\{0,\dots,n-1\}}(y)=\tau^{\{0,\dots,n-1\}}(w)$ and therefore $\spa(T,\tau^{\{0,\dots,n-1\}}(y),\epsilon/2)=\spa(T,\tau^{\{0,\dots,n-1\}}(w),\epsilon/2)$. Thus \Cref{eq:long sum} can be continued as
\(
\phantom{a}\leq E_{\epsilon}\cdot \sum_{w\in L_{n,s}(Y)}\spa(T,\tau^{\{0,\dots,n-1\}}(w),\epsilon/2)
=
E_{\epsilon}A_n(\epsilon/2).
\)
This finishes the proof. 
\end{proof} 
\subsection{The slow entropy of $S\rtimes_\tau T$}
In this section we compute the slow entropy of $S\rtimes_\tau T$ with the scales from \Cref{def:scales}. The next result proves \Cref{thm:entropy} as the system $(X,T)$ is arbitrary.
\begin{proposition}[Proof of \Cref{thm:entropy}]\label{prop:entropy}
The upper and lower slow entropy of $S\rtimes_\tau T$ with respect to $a_n(t)$ is equal to $h_{top}(T)$. That is,
\[\lowerent_{\textbf{a}}(S\rtimes_\tau T)=\upperent_{\textbf{a}}(S\rtimes_\tau T)=h_{top}(T)\]
\end{proposition}
\begin{proof}
    We start by proving the inequality $
    \upperent_{\textbf{a}}(S\rtimes_\tau T)\leq h_{top}(T)$. This is clearly true if $h_{top}(T)=\infty$, so we assume $h_{top}(T)<\infty$. Let $t>0$ be arbitrary with $h_{top}(T)<t$. We will prove that $\upperent_{\textbf{a}}(S\rtimes_\tau T)\leq t$. According to the definition of upper slow entropy, we must show that for arbitrarily small $\epsilon$ we have
    \[\limsup_{n\to\infty} \frac{\spa(S\rtimes_\tau T,\{0,\dots,n-1\},\epsilon)}{a_n(t)}=0.\]
    By \Cref{prop:the-set-A} it suffices to prove that for arbitrarily small $\epsilon$
    \begin{equation}\label{eq:limsup}\limsup_{n\to\infty} \frac{A_n(\epsilon)}{a_n(t)}=0.\end{equation}
    Choose $t'$ with $h_{top}(T)<t'<t$. \Cref{estimate} shows that there is an arbitrarily small $\epsilon$, fixed from now on, with the following property. There is $n_0\in\N$ such that \[\spa(T,F,\epsilon)\leq e^{C_m(F)|F|t'}\]
    for every $F\Subset\Z$ with $m$-bounded gaps and $|F|\geq n_0$. Recall that $r_n(w)$ and $q_n(w)$ are defined in \Cref{def:scales}. If we apply the inequality above to $F=\tau^{\{0,\dots,n-1\}}(w)$ for some $w\in L_{n,s}(Y)$ with $r_n(w)\geq n_0$, we obtain
    \begin{equation}\label{eq:estimate-1} \spa(T,\tau^{\{0,\dots,n-1\}}(w),\epsilon)\leq e^{q_n(w)t'}.\end{equation}
    We will prove that for this $\epsilon$ the limsup in \Cref{eq:limsup} is 0. For this purpose we fix an arbitrary $\gamma>0$ and we prove that the limsup is at most a multiple of $\gamma$. We will now choose some parameters, the reason of these choices will become clear later. 
    
    Since $t-t'>0$, the function $1/e^{n(t-t')}$ converges to $0$ as $n\to\infty$, and we can find $n_1\geq n_0$ such that for all $n\geq n_1$ we have $1/e^{n(t-t')}<\gamma$. For a word $w\in L_{n,s}(Y)$ with $r_n(w)\geq n_1$, by \Cref{eq:estimate-1} we have \[\spa(T,\tau^{\{0,\dots,n-1\}}(w),\epsilon)/e^{q_n(w)t}\leq 1/{e^{q_n(w)(t-t')}}<\gamma.\] The relevant consequence for us is 
    \begin{equation}\label{n1-definition}
    \frac{\spa(T,\tau^{\{0,\dots,n-1\}}(w),\epsilon)}{e^{q_n(w)t}}<\gamma, \ \ \ r_n(w)\geq n_1.
    \end{equation}
    After choosing $n_1$, we let $N\in\N$ be equal to $\spa(T,I,\epsilon)$, where $I\subset\Z$ is an interval with length $2n_1m$. Therefore when  $F\Subset\Z$ is an arbitrary set with $m$-bounded gaps and $|F|\leq n_1$, a translate of $F$ is contained in $I$ and therefore $\spa(T,F,\epsilon)\leq N$ (\Cref{monotony-properties-of-spa}).  
    
    After choosing $N$ we observe that the function $N/e^{nt}$ converges to $0$ as $n\to\infty$, and then we can find $n_2\geq n_1$ with the property 
    \begin{equation}\label{property-of-n2}
    N/e^{nt}<\gamma, \ \ \ \forall n\geq n_2.\end{equation}
    We now use the fact that $(Y,S,\tau)$ satisfies  \Cref{def:good}. This property ensures that we can choose $n_3$ such that for every $n\geq n_3$, the proportion of words $w\in L_{n,s}(Y)$ with $r_n(w)\geq n_2$ is at least $1/\ell$. We claim that \begin{equation}\label{bound-A_n/a_n}\frac{A_n(\epsilon)}{a_n(t)}<\gamma(\ell+2), \ \ \ \forall n\geq n_3.\end{equation}
    Recall that $A_n(\epsilon)$ is equal to the sum of $\spa(T,\tau^{\{0,\dots,n-1\}}(w),\epsilon)$, with $w$ ranging over $L_{n,s}(Y)$. We will separate this sum in the cases $r_n(w)<n_1$ and $r_n(w)\geq n_1$. If $r_n(w)\geq n_1$ then by \Cref{eq:estimate-1} we have  
    $\spa(T,\tau^{\{0,\dots,n\}}(w),\epsilon)\leq \gamma e^{q_n(w)t}$. Therefore \begin{align*}\sum_{\substack{w\in L_{n,s}(Y)\\ r_n(w)\geq n_1}}\spa(T,\tau^{\{0,\dots,n-1\}}(w),\epsilon)&\leq \gamma\sum_{\substack{w\in L_{n,s}(Y)\\ r_n(w)\geq n_1}} e^{q_n(w)t} \\ &\leq \gamma a_n(t).
    \end{align*}
    We now consider the  words $w$ with $r_n(w)< n_1$.  By definition of $N$ in this case we have $\spa(T,\tau^{\{0,\dots,n-1\}}(w),\epsilon)\leq N$. Taking the sum over words $w$ with $r_n(w)< n_1$ we have
    \[\sum_{\substack{w\in L_{n,s}(Y)\\ r_n(w)< n_1}}\spa(T,\tau^{\{0,\dots,n-1\}}(w),\epsilon)\leq |L_{n,s}(Y)|N.
    \]
    Since the proportion of words in $L_{n,s}(Y)$ with $r_n(w)\geq n_2$ is at least $1/\ell$, we have $|L_{n,s}(Y)|\leq |\{w\in L_{n,s}(Y) : r_n(w)\geq n_2\}|\cdot (\ell+1)$. Therefore the inequality above can be continued as
    \begin{align*}
    &\leq |\{w\in L_{n,s}(Y) : q_n(w)\geq n_2\}| (\ell+1) N\\&=(\ell+1)\sum_{\substack{w\in L_{n,s}(Y)\\ r_n(w)\geq  n_2}} N
    \end{align*}
    Our choice of $n_2$ ensures that for every $w$ with $r_n(w)\geq n_2$ we have $N<\gamma e^{q_n(w)t}$ (by \Cref{property-of-n2} and because  $r_n(w)\leq q_n(w)$). Therefore we can bound the previous sum as 
    \(
    \leq (\ell+1) \sum_{\substack{w\in L_{n,s}(Y)\\ r_n(w)\geq  n_2}} \gamma e^{q_n(w)t}
    \leq (\ell+1)\gamma a_n(t)
    \)
    We have proved that for all $n\geq n_3$ we have  $A_n(\epsilon)\leq (\ell+2) \gamma a_n(t)$ and therefore $A_n(\epsilon)/a_n(t)\leq (\ell+2) \gamma$. This shows that the limsup of $A_n(t)/a_n(t)$ as $n\to\infty$ is at most $(\ell+2) \gamma$. Since $\gamma$ is arbitrary, it follows that this limsup us $0$. Thus we have proved \Cref{eq:limsup}. Since this holds for  $\epsilon$ arbitrarily small, we have proved our claim that $\upperent_{\textbf{a}}(S\rtimes_\tau T)\leq h_{top}(T)$. 
    We will now prove that $h_{top}(T)\leq \lowerent_{\textbf{a}}(T)$. This is clearly true if $h_{top}(T)=0$, so we assume now $h_{top}(T)\in (0,\infty]$. We fix an arbitrary $t$ with $0<t<h_{top}(T)$ and we prove that $t\leq \lowerent_{\textbf{a}}(S\rtimes_\tau T)$. According to the definition of lower slow entropy we must prove that for arbitrarily small $\epsilon$ we have
    \[\liminf_{n\to\infty} \frac{\spa(S\rtimes_\tau T,\{0,\dots,n-1\},\epsilon)}{a_n(t)}>0.\]
    Thanks to \Cref{prop:the-set-A}, it suffices to prove that for arbitrarily small $\epsilon$ we have \begin{equation}\label{liminf}\liminf_{n\to\infty} \frac{A_n(\epsilon)}{a_n(t)}>0.\end{equation}
    Choose $t'$ such that $t<t'<h_{top}(T)$. By \Cref{estimate}, for $\epsilon$ arbitrarily small we have that
    \begin{equation}\label{lower-bound-spa}
    \spa(T,F,\epsilon)\geq e^{C(F)|F|t'}, \ \ \ F\Subset \Z.
    \end{equation}
    We will prove that for this $\epsilon$, \Cref{liminf} holds. For each $n\in\N$ we define $g(n)$ as 
    \[g(n)=\min \{\frac{\spa(T,\tau^{\{0,\dots,n-1\}}(w),\epsilon)}{e^{q_n(w)t}} : w\in L_{n,s}(Y)\}\]
    We also let $w_n$ be a word realizing this minimum, that is \[g(n)=\frac{\spa(T,\tau^{\{0,\dots,n-1\}}(w_n),\epsilon)}{e^{q(w_n)t}}.\]
    A direct computation shows that  $g(n)\cdot a_n(t)\leq A_n(\epsilon)$. Therefore $A_n(\epsilon)/a_n(t)\geq g(n)$ for all $n$, and to prove \Cref{liminf} it suffices to show
    \[\liminf_{n\to\infty} g(n)>0.\]
    By \Cref{lower-bound-spa},  for each $n$ we have $\spa(T,\tau^{\{0,\dots,n-1\}}(w_n),\epsilon)\geq e^{q(w_n)t'}$, so
\(\liminf_{n\to\infty} g(n)=\liminf_{n\to\infty} \frac{\spa(T,\tau^{\{0,\dots,n-1\}}(w_n),\epsilon)}{e^{q(w_n)t}}\geq \liminf_{n\to\infty} \frac{e^{q(w_n)t'}}{e^{q(w_n)t}}
=\liminf_{n\to\infty}e^{q(w_n)(t'-t)}
\)
Since $q(w_n)\geq 1$ for each $n$ and $t-t'>0$, this liminf is nonzero.  Thus we have proved \Cref{liminf}.

Finally, a general fact is that $\lowerent_{\textbf{a}}(S\rtimes_\tau T)\leq \upperent_{\textbf{a}}(S\rtimes_\tau T)$. Thus we have
\[h_{top}(T)\leq \lowerent_{\textbf{a}}(S\rtimes_\tau T)\leq \upperent_{\textbf{a}}(S\rtimes_\tau T)\leq h_{top}(T).\]
This proves that $\lowerent_{\textbf{a}}(S\rtimes_\tau T)$ and $\upperent_{\textbf{a}}(S\rtimes_\tau T)$ are both equal to $h_{top}(T)$.
\end{proof}
In the next result we prove \Cref{thm:slow-entropy}. We remark that the proof follows similar ideas to the previous one. Instead of \Cref{estimate} we invoke abstract estimates that follow from the definition of $\lowerent_{\textbf{b}}(T)$ and $\upperent_{\textbf{b}}(T)$. The hypothesis that $\tau$ attains values in $\{-1,0,1\}$ ensures that the sets $\tau^{\{0,\dots,n-1\}}(w)$ are always intervals, and this simplifies some considerations. Recall that $b_n(t)$ is an arbitrary scale, and $c_n(t)$ depends on $b_n(t)$ and is defined in \Cref{def:scales}.
\begin{proposition}[Proof of \Cref{thm:slow-entropy}]
Suppose that for every $y\in Y$ we have $\tau(y)\in \{-1,0,1\}$. Then
\[\upperent_{\textbf{c}}(S\rtimes_\tau T)\leq \upperent_{\textbf{b}}(T)\]
\[\lowerent_{\textbf{c}}(S\rtimes_\tau T)\geq \lowerent_{\textbf{b}}(T)\]
\end{proposition}
\begin{proof}
Observe that the assumption on $\tau$ implies that for every $n$ and $w\in L_{n,s}(Y)$ the set $\tau^{\{0,\dots,n-1\}}(w)$ is an interval. Since its cardinality is $r_n(w)$, we can always shift this interval and obtain that $\spa(T,\tau^{\{0,\dots,n-1\}}(w),\epsilon)$ is equal to $\spa(T,\{0,\dots,r_n(w)-1\},\epsilon)$ (\Cref{monotony-properties-of-spa}). We will use this repeatedly in this proof. 

We start by proving $\upperent_{\textbf{c}}(S\rtimes_\tau T)\leq \upperent_{\textbf{b}}(T)$. This is clearly true if $\upperent_{\textbf{b}}(T)=\infty$, so we assume $\upperent_{\textbf{b}}(T)<\infty$ in what follows. We take an arbitrary $t$ with $\upperent_{\textbf{b}}(T)<t$ and we prove $\upperent_{\textbf{c}}(S\rtimes_\tau T)\leq t$. According to the definition of upper slow entropy we must prove that for arbitrarily small $\epsilon$ we have
\[\limsup_{n\to\infty} \frac{\spa(S\rtimes_\tau T,\{0,\dots,n-1\},\epsilon)}{c_n(t)}=0.\]
Thanks to \Cref{prop:the-set-A} it suffices to prove that for arbitrarily small $\epsilon$ we have
\begin{equation}\label{limsup-slow}
    \limsup_{n\to\infty} \frac{A_n(\epsilon)}{c_n(t)}=0.
\end{equation}
Since $\upperent_{\textbf{b}}(T)<t$, we have $\upperent_{\textbf{b}}(T,\epsilon)<t$ for arbitrarily small $\epsilon$. Fix $\epsilon$ with this property. We will prove that the limsup in \Cref{limsup-slow} is 0 for this $\epsilon$. For this we take an arbitrary $\gamma$ and we prove that the limsup is at most a multiple of $\gamma$. 

Since $\upperent_{\textbf{b}}(T,\epsilon)<t$, the definition of upper slow entropy shows that the limsup of $\spa(T,\{0,\dots,n-1\},\epsilon)/b_n(t)$ is $0$.  Therefore we can take $n_1$ such that 
\begin{equation}\label{n_1-slow}
    \frac{\spa(T,\{0,\dots,n-1\},\epsilon)}{b_n(t)}<\gamma, \ \ \ \forall n\geq n_1.
\end{equation}
After choosing $n_1$ we let $N=\spa(T,\{0,\dots,n_1-1\},\epsilon)$. By definition of scale we have that $n\mapsto b_n(t)$ tends to infinity, so we can choose $n_2$ with the property 
\begin{equation}\label{n_2-slow}N/b_n(t)\leq \gamma, \ \ \ \forall n\geq n_2.\end{equation}
Since $(Y,S,\tau)$ satisfies \Cref{def:good}, it is possible to choose $n_3$ such that for every $n\geq n_3$, the proportion of words $w\in L_{n,s}(Y)$ with $r_n(w)\geq n_2$ is at least $1/\ell$. 

Let $n\geq n_3$.  \Cref{limsup-slow} shows that 
\begin{align*}
\sum_{\substack{w\in L_{n,s}(Y)\\ r_n(w)\geq n_1}}\spa(T,\tau^{\{0,\dots,n-1\}}(w),\epsilon)&\leq \gamma\sum_{\substack{w\in L_{n,s}(Y)\\ r_n(w)\geq n_1}} b_{q_n(w)}(t) \\ &\leq \gamma c_n(t).
\end{align*}
Using \Cref{n_2-slow}, the same argument used in the proof of \Cref{prop:entropy} shows that
\begin{align*}
\sum_{\substack{w\in L_{n,s}(Y)\\ r_n(w)< n_1}}\spa(T,\tau^{\{0,\dots,n-1\}}(w),\epsilon)&\leq (\ell+1)\gamma c_n(t).
\end{align*}
We obtain that   $A_n(\epsilon)\leq (\ell+2)\gamma c_n(t)$ for all $n\geq n_3$, and therefore the limsup of $A_n(\epsilon)/c_n(t)$ is at most $(\ell+2)\gamma$. 
Since $\gamma$ is arbitrary and $\ell$ only depends on $\tau$, this limsup must be $0$. Thus we have proved \Cref{limsup-slow}, and consequently that $\upperent_{\textbf{c}}(S\rtimes_\tau T)\leq \upperent_{\textbf{b}}(T)$. 

We will now prove that $\lowerent_{\textbf{c}}(S\rtimes_\tau T)\geq \lowerent_{\textbf{b}}(T)$. This is certainly true if $\lowerent_{\textbf{b}}(T)=0$. Assume now that $\lowerent_{\textbf{b}}(T)\in (0,\infty]$, and let $t$ with $0<t<\lowerent_{\textbf{b}}(T)$. We will prove that $t\leq \lowerent_{\textbf{c}}(S\rtimes_\tau T)$. According to the definition of $\lowerent$ we must show that for arbitrarily small $\epsilon$ we have
\[\liminf_{n\to\infty} \frac{\spa(S\rtimes_\tau T,\{0,\dots,n-1\},\epsilon)}{c_n(t)}> 0.\]
Thanks to \Cref{prop:the-set-A}, it suffices to show that for arbitrarily small $\epsilon$ we have
\begin{equation}\label{liminf-slow}\liminf_{n\to\infty} \frac{A_n(\epsilon)}{c_n(t)}> 0\end{equation}
Since $t$ is chosen with $t<\lowerent_{\textbf{b}}(T)$, for all $\epsilon$ small enough we have $t<\lowerent_{\textbf{b}}(T,\epsilon)$. We fix $\epsilon$ with this property and we prove that \Cref{liminf-slow} holds for this $\epsilon$. For each $n\in\N$ we define $g(n)$ as 
    \[g(n)=\min \{\frac{\spa(T,\tau^{\{0,\dots,n-1\}}(w),\epsilon)}{b_{r_n(w)}(t)} : w\in L_{n,s}(Y)\}\]
We also let $w_n$ be a word realizing this minimum, that is \[g(n)=\frac{\spa(T,\tau^{\{0,\dots,n-1\}}(w_n),\epsilon)}{b_{r_n(w)}(t)}.\]
A direct computation shows that $g(n)\cdot c_n(t)\leq A_n(\epsilon)$. Therefore $A_n(\epsilon)/a_n(t)\geq g(n)$, and to prove \Cref{liminf} it suffices to show
\[\liminf_{n\to\infty} g(n)>0.\]
Since $t$ is chosen with $t<\lowerent_{\textbf{b}}(T,\epsilon)$, we have by definition of $\lowerent$ that the liminf of $\spa(T,\{0,\dots,n-1\},\epsilon)/b_n(t)$ is positive. Choose $\gamma>0$ to be strictly smaller than the liminf of $S(T,\{0,\dots,n-1\},\epsilon)/b_n(t)$. Therefore we can find $n_0$ such that for every $n\geq n_0$ we have $\spa(T,\{0,\dots,n-1\},\epsilon)/b_n(t)>\gamma$. We now let $\eta$ be equal to the minimum of  $\spa(T,\{0,\dots,n-1\},\epsilon)/b_n(t)$, with $n$ ranging over $\{0,\dots,n_0-1\}$. Observe that $\eta>0$. Since $g(n)$ is equal to $\spa(T,\{0,\dots,r_n(w_n)-1\},\epsilon)/b_{r_n(w_n)}(t)$, it follows that $g(n)\geq \gamma$ whenever $r_n(w_n)\geq n_0$ and $g(n)\geq \eta$ whenever $r_n(w_n)\leq n_0$. In particular $g(n)\geq \min\{\gamma,\eta\}$ for every $n$, so the liminf of $g(n)$ is  positive.  
\end{proof}
\begin{remark}\label{non-invertible}
    If $(Y,S)$ is a $\N$-subshift instead of a $\Z$-subshift then all the arguments presented here can be easily adapted. One only needs to redefine $L_{n,s}(Y)$ as $L_{\{0,\dots,s+n-1\}}(Y)$ instead of $L_{\{-s,\dots,s+n-1\}}(Y)$  (as in \Cref{subsec:subshifts}). 
\end{remark}
\section{Examples}
Given a tuple $(Y,S,\tau)$ of a subshift $S\colon Y\to Y$ and a continuous cocycle $\tau\colon Y\to \Z$, we consider the family of skew products $S\rtimes_\tau T$, where $(X,T)$ is an arbitrary invertible topological dynamical system. In the next result we prove a sufficient condition on $(Y,S,\tau)$ that implies that all systems of the form $S\rtimes_\tau T$ have the same topological  entropy. This justifies the application of other invariants different from topological entropy. 



\label{sec:examples}
\begin{proposition}\label{criterion-zero-entropy}
    Let $(Y,S)$ be a subshift,  and let $\tau\colon Y\to\Z$ be a continuous function such that for all $y\in Y$ we have  
    \[\lim_{n\to\infty}\frac{1}{n}\sum_{i=0}^{n-1}\tau(S^iy)= 0,\]
    and furthermore this convergence is uniform over $y\in Y$. Then for every invertible topological dynamical system $(X,T)$ we have $h_{top}(S\rtimes_\tau T)=h_{top}(S)$.
\end{proposition}
\begin{proof}[Proof sketch]
The nontrivial inequality is $h_{top}(S\rtimes_\tau T)\leq h_{top}(S)$. Let $\eta\in(0,1)$ be arbitrary. The hypothesis ensures that there is $n_0$ such that for all $n\geq n_0$ we have $\frac{1}{n}\sum_{i=0}^{n-1}\tau(S^iy)<\eta/2$ for all $y\in Y$. Using this and \Cref{prop:the-set-A} one obtains that for all $n$ large enough
\[\spa(S\rtimes_\tau T,\{0,\dots,n-1\},\epsilon)\leq E\cdot |L_{n,s}(Y)|\cdot \spa(T,\{0,\dots,\lceil \eta n\rceil-1,\epsilon/2).\]
Here $E$ is a constant independent of $n$. Applying logarithm, dividing by $n$, and taking the limsup one obtains $h_{top}(S\rtimes_\tau T)\leq h_{top}(S)+\eta h_{top}(T)$. As $\eta\in(0,1)$ is arbitrary it follows that   $h_{top}(S\rtimes_\tau T)\leq h_{top}(S)$ as claimed.
\end{proof}
We remark that the condition in the last result is met when $(Y,S)$ is uniquely ergodic and $\tau$ has zero mean with respect to the unique invariant measure \cite[Theorem 4.10]{einsiedler_ergodic_2011}.

The next result provides a condition that implies that $\tau$ is $1$-unbounded.  
\begin{proposition}\label{criterion-1-unbounded}
    Let $(Y,S)$ be a subshift and let $\tau\colon Y\to \Z$ be a continuous map such that $n\mapsto\sum_{i=0}^{n-1}\tau(S^iy)$ is unbounded for every $y\in Y$. Then $\tau$ is $1$-unbounded.
\end{proposition}
\begin{proof}
    We claim that for every $N\in\N$ we have \[\liminf_{n\to\infty} \frac{|\{w\in L_{n,s}(Y) : r_n(w)\geq N\}|}{|L_{n,s}(Y)|}=1\]
    Suppose for a contradiction that this is false. Then there is $N\in\N$ such that the liminf above is strictly smaller than $1$. We can find an increasing sequence $(n_k)_{k\in\N}$ of natural numbers, and words $w_k\in L_{n_k,s}(Y)$ such that $r_{n_k}(w_{n_k})< N$. For each $w_k$ we can choose $y_k\in Y$ with $w_k\sqsubset y_k$. Since $Y$ is compact the sequence $(y_k)_{k\in\N}$ has an accumulation point $y^*$. We claim that $r_n(y)<N$ for all $n\in\N$. Indeed, for arbitrary $n$ we can find $t\in\N$ large enough so that $r_n(y)$ depends on the restriction of $y$ to $\{-s,\dots,t+s-1\}$. Then we can find $k$ large enough so that $n_k>t$ and such that $y_k$ and $y$ are $2^{-t-s}$ close. This implies that they have the same restriction to $\{-t-s,\dots,t+s\}$ so in particular $r_n(y^{\ast})=r_n(y_k)$. We have $r_n(y_k)\leq r_{n_k}(y_k)$ as $n\leq n_k$, and $r_{n_k}(y_k)< N$ by definition of $y_k$ and $w_k$. Thus $r_n(y^{\ast})<N$ for all $n$. This implies that $n\mapsto\sum_{i=0}^{n-1}\tau(S^iy^\ast)$ is bounded, a contradiction to the hypothesis. 
\end{proof}

Recall that a continuous cocycle $\tau\colon Y\to\Z$ is called a coboundary when  $\tau=g-g\circ S$ for some continuous function $g\colon Y\to\R$. The  Gottschalk-Hedlund theorem \cite[Theorem 2.9.4]{katok_introduction_1995} states that if $\tau$ is not a cocycle then  $n\mapsto\sum_{i=0}^{n-1}\tau(S^iy)$ is unbounded for every $y\in Y$. From this an the previous result we obtain the following. 
\begin{proposition}\label{gottschalk-non-coboundary}
    Let $(Y,S)$ be a minimal subshift and let $\tau\colon Y\to\Z$ be a continuous map which is not a coboundary. Then $\tau$ is $1$-unbounded.
\end{proposition}
This result provides a large class of systems where \Cref{corollary} can be applied. We now review one specific example in detail. 
\begin{example}[Deterministic random walk]\label{example-1}
    Let $\mathbb T=\R/\Z$ and let $R_{\alpha}\colon\mathbb T \to \mathbb T$, $x\mapsto x+\alpha$ be an irrational rotation. Let $\tau\colon [0,1)\to \{-1,1\}$ be the map that equals $1$ over $[0,1/2)$ and $-1$ otherwise. The Sturmian subhift of type $(1/2,\alpha)$, which we denote $(Y_\alpha,S_\alpha)$, is the smallest subshift on alphabet $\{-1,1\}$ containing all sequences of the form $(\tau(x+n\alpha))_{n\in\Z}$, $x\in \mathbb T$. This subshift is minimal, uniquely ergodic, has zero topological entropy, and  is measurably isomorphic to $(\mathbb{T},R_\alpha)$ with the Lebesgue measure. 
    
    Consider the continuous cocycle $\tau_{\alpha}\colon Y\to\{-1,1\}$ given by $y\mapsto y(0)$, and consider the family of skew products $S_{\alpha}\rtimes_{\tau_\alpha}T$, where $T$ is an arbitrary topological dynamical system. \Cref{criterion-zero-entropy} shows that all these systems have zero topological entropy (note that $\tau_\alpha$ has zero mean with respect to the unique invariant measure of $S_\alpha$, so by Birkhoff's ergodic theorem the hypothesis in \Cref{criterion-zero-entropy} is met).

    It follows from a classic result of Kesten \cite{kesten_conjecture_1966} that $\tau_\alpha$ is not a coboundary. By \Cref{gottschalk-non-coboundary}, it follows that $\tau_\alpha$ is $1$-unbounded. Then \Cref{thm:entropy} shows that there is a scale $\{a_n(t)\}_{n\in\N,t>0}$ such that $\upperent_{\textbf{a}}(S_{\alpha}\rtimes_{\tau_\alpha}T)=h_{top}(T)$ for every topological dynamical system $(X,T)$. 

    Let us note that the same conclusion can not be obtained by mean of topological entropy dimension, another invariant which has shown to be useful for zero-entropy systems. Indeed, in \cite{dou_entropy_2022} the authors show that there is $r\in\R$ associated to $\alpha$ such that the entropy dimension of $S_{\alpha}\rtimes_{\tau_{\alpha}} T$ equals $r$ for every $T$ with positive topological entropy. This is proved under some hypothesis on $\alpha$, and holds for an uncountable and dense set of $\alpha$'s. 
\end{example}
A simple modification of the last example provides a family of systems where topological sequence entropy has no distinguishing power, but slow entropy does.  
\begin{example}\label{example:slow-entropy-and-sequence-entropy}
    Consider the tuple $(Y_{\alpha},S_{\alpha},\tau_{\alpha})$ be from  \Cref{example-1}, let $Y=Y_{\alpha}\times \{-1,1\}^\Z$, let $S$ be the (componentwise) shift map on $Y_{\alpha}\times \{-1,1\}^\Z$, and define $\tau\colon Y\to\{-1,1\}$ by $\tau(x,y)=\tau_{\alpha}(x)$. We consider the family of systems of the form $S\rtimes_\tau T$, where $(X,T)$ is an arbitrary invertible topological dynamical system.

    We claim that for every sequence $A$ the corresponding invariant of topological sequence entropy  $h_{top}^A(\cdot)$ is constant on this family of systems. Indeed, by \Cref{criterion-zero-entropy} we have $h_{top}(S\rtimes_{\tau}T)=h_{top}(S)=\log(2)$ for every choice of $T$. Then  part (3) of \Cref{goodman-v2} shows that
    \[h^A_{top}(S\rtimes_{\tau}T)=K(A)h_{top}(S\rtimes_\tau T)=K(A)\log(2).\] 
    That is, $h^A_{top}(S\rtimes_{\tau}T)$ has the same value for any choice of $T$. 
    
    On the other hand, \Cref{criterion-1-unbounded} shows that $\tau$ is $1$-unbounded, and then by  \Cref{thm:entropy} there is 
    a scale $\{a_n(t)\}_{n\in\N,t>0}$  such that $\upperent_{\textbf{a}}(S\rtimes_{\tau}T)=h_{top}(T)$ for all $T$.  
\end{example}
In the next result we will verify that our results can be applied to the classical $T,T^{-1}$ system. That is, the shift over $\{-1,1\}^\Z$ and $\tau$ given by $\tau(y)=y(0)$.  
\begin{proposition}\label{fullshift-unbounded}
    Let $Y=\{-1,1\}^\Z$, let $S\colon Y\to Y$ be the shift map, and let $\tau\colon Y\to\Z$ be given by $\tau(y)=y(0)$. Then $\tau$ is $1$-unbounded.
\end{proposition}
\begin{proof}[Proof sketch]
    Let $\mu$ be the uniform measure on $\{-1,1\}^\Z$. It is a well-known fact that $n\mapsto \sum_{i=0}^{n-1}\tau(S^iy)$ is unbounded for a subset of $\{-1,1\}^\Z$ of full measure. This follows from the standard central limit theorem, or the law of iterated logarithm, or can be seen as a classic fact about random walks on $\Z$. Furthermore, for any $N\in\N$ the proportion of words $w\in L_n(Y)$ with $r_n(w)\geq N$ is equal to  $\mu(\{y\in Y : r_n(y)\geq N\})$, which converges to $1$ as $n\to\infty$ by the previous observation. 
\end{proof} 
In contrast to the measure-theoretic case, it is easy to observe\footnote{In this case one easily derives from \Cref{prop:the-set-A} that $\max\{h_{top}(T),h_{top}(S)\}\leq h_{top}(S\rtimes_\tau T)\leq h_{top}(T)+h_{top}(S)$.} that with this system in the base, the topological entropy of $S\rtimes_\tau T$ is not independent of $T$, and thus applying \Cref{thm:entropy} is less justified. However, the conclusion of \Cref{thm:slow-entropy} is still interesting if we consider skew products $S\rtimes_\tau T$ where  $T$ ranges over a class of systems with zero topological entropy. In this manner one can find other families of systems where sequence entropy is a trivial invariant (with the same argument from \Cref{example:slow-entropy-and-sequence-entropy}) but slow entropy is not (applying \Cref{thm:slow-entropy}).
\begin{example}
We will consider the class of transformations studied in \cite{kanigowski_slow_2019}, see this reference for undefined terms. Let $X=G/\Gamma$ be the quotient of a connected Lie group $G$ by a co-compact lattice $\Gamma$. Let $\mathfrak{g}$ be a Lie algebra for $G$, let $U\in\mathfrak{g}$ be an ad-quasi-unipotent element, and consider the left multiplication map $T\colon X\to X$ given by  $T(x)=\exp(g)x\Gamma$. The topological slow entropy of $(X,T)$ with scale $b_n(t)=n^t$ is computed in \cite[Theorem 1.10]{kanigowski_slow_2019}. It is an integer number and furthermore it is proved that $\upperent_{\textbf{b}}(T)=\lowerent_{\textbf{b}}(T)$. See the reference for an explicit formula for $\upperent_{\textbf{b}}(T)$ in terms of the chain structure of $U$. We also note that Theorem 1.10 in \cite{kanigowski_slow_2019} refers to the slow entropy of the flow associated to $U$, but the  transformation $T$ has the same slow entropy by \Cref{slow-entropy-of-flows-and-transformations} below. 

   Let $S$ be the shift map on $Y=\{-1,1\}^\Z$ and let $\tau\colon Y\to\Z$, $\tau(y)=y(0)$ for all $y\in Y$. We consider the family of skew products $S\rtimes_\tau T$, with $T$ as in the previous paragraph. It is easy to check that $h_{top}(S\rtimes_\tau T)=h_{top}(S)$ for any choice of $T$, so neither entropy nor sequence entropy provide information on this class of systems (see the argument in \Cref{example:slow-entropy-and-sequence-entropy}). On the other hand, by \Cref{fullshift-unbounded} and \Cref{thm:slow-entropy} there is a scale $c_n(t)$ such that $\upperent_{\textbf{c}}(S\rtimes_\tau T)=\upperent_{\textbf{b}}(T)$ for every choice of $T$.   
\end{example}
In the next result we verify that the slow entropy of a flow equals the slow entropy of its time one map. By a flow we mean a family $(T^{t})_{t\in \R}$ of homeomorphisms of  a compact metric space $X$ such that $(x,t)\mapsto T^{t}(x)$ is continuous as a function of $x\in X$ and $t\in\R$, and which obeys the rule $T^{t+s}(x)=T^{t}(T^{s}(x))$, $s,t\in \R$, $x\in X$. The definitions from \Cref{sec:preliminaries} can be directly adapted to flows by replacing $\{0,\dots,n-1\}$ by a compact interval $I\subset \R$. In this manner one defines Bowen metrics $d^B_I$, the values  $\spa(T,I,\epsilon)$, and the topological slow entropy of a flow. The reader is referred to \cite[Section 1.1]{kanigowski_slow_2019} for more details.
\begin{proposition}\label{slow-entropy-of-flows-and-transformations}
The topological slow entropy of a continuous flow on a compact metric space equals the topological slow entropy of its time one map, for any scale.  
\end{proposition}
\begin{proof}
Let $T\colon X\times \R\to X$, $(x,t)\mapsto T^{t}(x)$ be a continuous flow on the compact metric space $X$, and let  $T^1$ be the time-1 map of $T$. It is clear that
\[\spa(T^1,\{0,\dots,n-1\},\epsilon)\leq \spa(T,[0,n-1],\epsilon),\ \ \  \epsilon>0, n\in\N.\]
We claim that for every $\epsilon>0$ we can find $\epsilon'>0$ smaller than $\epsilon$ and with the property
\[\spa(T,[0,n-1],\epsilon)\leq \spa(T,\{0,\dots,n-1\}, \epsilon'),\ \ \  n\in\N.\]
The claim follows directly from these two inequalities and the definitions.

Let $\epsilon>0$, and consider the restriction of $(T^t)_{t\in\R}$ to a function $X\times [0,1]\to X$. Since $T^t(x)$ is continuous as a function of $t$ and $x$, and $X\times [0,1]$ is compact, the restriction is uniformly continuous, and thus we can find $\delta>0$ such that for every $x,x'\in X$, $t,t'\in [0,1]$, we have that $\max\{d(x,x'),|t-t'|\}<\delta$ implies $d(T^t(x),T^{t'}(x'))<\epsilon$. After doing this, choose elements $0=s_1<s_2<\dots<s_k=1$ such that $|s_i-s_{i+1}|<\delta$ for every $i\in\{1,\dots,k\}$. For every $s_i$ we consider the time $s_i$ map $T^{s_i}\colon X\to X$. For every $i=1,\dots,k$ the map $T^{s_i}\colon X\to X$ is uniformly continuous, and thus we can find  $\eta_i$ with the property that for all $x,y\in X$, $d(x,y)<\eta_i$ implies $d(T^{s_i}(x),T^{s_i}(y))<\delta$. 

We claim that  $\epsilon'=\min\{\eta_1,\dots,\eta_k,\epsilon/2\}$ satisfies $\spa(T,[0,n-1],\epsilon)\leq \spa(T,\{0,\dots,n-1\}, \epsilon')$ for all $n\in\N$. It suffices to take an arbitrary pair of elements $x,y\in X$ with $d^B_{\{0,\dots,n-1\}}(x,y)<\epsilon'$, and prove $d^B_{[0,n-1]}(x,y)<\epsilon$. We take an arbitrary element $s\in [0,n-1]$ verify that  $d(T^{s}(x),T^{s}(Y))<\epsilon$. For this we write $s$ as $s=j+s_i+\alpha$, where $i,j\in\{0,\dots,n-1\}$, $s_i\in [0,1]$, and $|\alpha|<\delta$. 
     Since $d^B_{\{0,\dots,n-1\}}(x,y)<\epsilon'$ by assumption, we have  $d(T^j(x),T^j(y))<\epsilon$ for every $j=1,\dots,k$, and this implies  $d(T^j(x),T^j(y))<\eta_i$. Since  $T^j(x),T^j(y)$ are $\eta_i$-close, their images by $T^{s_i}$, $T^{j+s_i}(x),T^{j+s_i}(y)$ must be $\delta$-close. Since $|\alpha|<\delta$, it follows that their images by $T^{\alpha}$, $T^{j+s_i+\alpha}(x),T^{j+s_i+\alpha}(y)$, are $\epsilon$ close. That is, $d(T^{s}(x),T^{s}(y))<\epsilon$. 
\end{proof}

\bibliographystyle{abbrv}
\bibliography{references}
\end{document}